\documentclass[a4paper]{amsart}
\usepackage{dettweiler-sabbah}

\begin{document}
\addtolength{\textheight}{-2\baselineskip}
\addtolength{\topmargin}{\baselineskip}
\frontmatter
\title{Hodge theory of the middle convolution}

\author[M.~Dettweiler]{Michael Dettweiler}
\address[M.~Dettweiler]{Lehrstuhl IV für Mathematik / Zahlentheorie\\
Department of Mathematics\\
University of Bayreuth\\
95440 Bayreuth\\
Germany}
\email{michael.dettweiler@uni-bayreuth.de}
\urladdr{http://www.zahlentheorie.uni-bayreuth.de/}

\author[C.~Sabbah]{Claude Sabbah}
\address[C.~Sabbah]{UMR 7640 du CNRS\\
Centre de Mathématiques Laurent Schwartz\\
École polytechnique\\
F--91128 Palaiseau cedex\\
France}
\email{sabbah@math.polytechnique.fr}
\urladdr{http://www.math.polytechnique.fr/~sabbah}

\thanks{The research of C.S.\ was supported by the grant ANR-08-BLAN-0317-01 of the Agence natio\-nale de la recherche.}



\begin{abstract}
We compute the behaviour of Hodge data by tensor product with a unitary rank-one local system and middle convolution by a Kummer unitary rank-one local system for an irreducible variation of polarized complex Hodge structure on a punctured complex affine line.
\end{abstract}
\maketitle

\tableofcontents
\mainmatter

\section*{Introduction}
Given an irreducible local system on a punctured projective line (over the field of complex numbers), the Katz algorithm \cite{Katz96} provides a criterion for testing whether this local system is physically rigid: this algorithm should terminate with a rank-one local system (\cf\S\ref{subsec:Katzalg}). This procedure is a successive application of tensor product with a rank-one local system and middle convolution with a Kummer local system. If the local monodromies of the rigid local system we start with have absolute value equal to one, then so do the eigenvalues of the terminal rank-one local system, which is then a unitary local system. Going the other direction in the algorithm, we conclude that the original local system underlies a variation of polarized complex Hodge structure (see also Theorem \ref{th:Simpson} due to C\ptbl Simpson for a more general argument). According to a general result of Deligne \cite{Deligne87} (\cf Prop\ptbl\ref{prop:Deligne}), any irreducible local system underlies at most one such variation, up to a shift of the filtration.

Our purpose in this article is to complement the Katz algorithm with the behaviour of various Hodge numerical data, when they are present, in order to be able to compute these Hodge data after each step of the algorithm. There are data of a local nature (Hodge numbers of the variation, Hodge numbers of vanishing cycles) and of a global nature (degrees of some Hodge bundles), and this set of data is enough to compute the same set of data after each step of the algorithm. One of the main tools is a general Thom-Sebastiani formula due to M\ptbl Saito \cite{MSaito90-11}.

As an application we compute the length of the Hodge filtration and the degree of the Hodge bundles of some local systems with $G_2$-monodromy. A restricted set of possible lengths may be obtained from \cite[Chap.\,IV]{G-G-K12} and, on a given example, the actual length may be deduced by eliminating various possibilities, as in \cite[\S9]{K-P12}, which treats an example of \cite{D-R10}. Here we do not use this a~priori knowledge.

On the one hand, by a direct application of the Katz algorithm with Hodge data, we compute the Hodge data of a physically rigid $G_2$-local system~$\cG$ on $\PP^1\setminus\nobreak \{x_1,x_2,x_3,x_4\}$ whose existence is proved in \cite[Th\ptbl1.3.2]{D-R10}. We find a Hodge filtration of length three, which is the minimal possible length for an irreducible local system with $G_2$-monodromy.

On the other hand, we compute the Hodge data for one of the orthogonally rigid local systems with $G_2$-monodromy classified in \cite{D-R12}. Since this local system is not physically rigid, we cannot directly apply our formulas, as in the previous case, and a supplementary computation is needed. The result is that the Hodge filtration has maximal length, equal to the rank (seven) of the local system. By applying the recent generalization of the potential automorphy criteria of \cite{B-G-G-T10} by Patrikis and Taylor~\cite{PatrikisTaylor} we are able to produce potentially automorphic representations $\Gal(\ov\QQ/\QQ)\to\GL_7(\QQ_\ell)$ for each prime number $\ell$.

\smallskip\noindent\emph{Acknowledgements.}
We thank Carlos Simpson for helpful comments and the referees for their careful reading of the manuscript and their useful suggestions.

\section{Preliminary results}
We review here various results on Fourier transform and middle convolution for holonomic modules on the affine line. These results exist in the literature, but are usually not directly proved for holonomic modules, but for $\ell$-adic sheaves (\cite{Katz96}), or assume a regular singularity condition. We give here a self-contained presentation for holonomic modules, which relies on various references (\cite{Malgrange91}, \cite{B-E04}, a letter of Deligne to Katz, dated October 2006, \cite{Bibi07a}, \cite{Arinkin08}).

\subsection{Local data of a holonomic $\cD(\Afu)$-module}\label{subsec:locnumdatahol}
Let $M$ be a holonomic $\Clt$-module and denote by $\bmx=\{x_1,\dots,x_r\}\subseteq \Afu$ the finite set of its singular points at finite distance. We will also consider $x_{r+1}=\infty$ as a singular point of $M$. Let $x_i$ be a singular point and let $x$ be a local coordinate at $x_i$ (\eg, $x=t-x_i$ if $i=1,\dots,r$ and $x=1/t$ if $i=r+1$). We set
\[
\wh M_{x_i}=\begin{cases}
\CC\lcr x\rcr\otimes_{\CC[t]}M&\text{if }i=1,\dots,r,\\
\CC\lpr x\rpr\otimes_{\CC[t]}M&\text{if }i=r+1.
\end{cases}
\]
These are holonomic $\CC\lcr x\rcr\langle\partial_x\rangle$-modules. Let $h(M)=\dim\CC(t)\otimes_{\CC[t]}M$ denote the generic rank of $M$ and $h_{x_i}(M)=\dim\CC\lpr x\rpr\otimes_{\CC[t]}M$ denote the generic rank of~$\wh M_{x_i}$. We have $h_{x_i}(M)=h(M)$ for each $i=1,\dots,r+1$.

\subsubsection*{The Levelt-Turrittin decomposition}
Let us introduce the \emph{elementary} modules $\El(\varphi,(R_\varphi,\nabla))$, which are the finite dimensional $\CC\lpr x\rpr$-vector spaces with connection obtained as follows:
\begin{itemize}
\item
$\varphi\in\CC\lpr x_p\rpr/\CC\lcr x_p\rcr$ is nonzero (and its pole order is denoted by \hbox{$q(\varphi)\in\NN^*$}), where $x_p$ is a new variable, and if we regard $\rho_p=x_p^p$ as a ramified covering $x_p\mto x=x_p^p$ of the formal disc, $\varphi$ cannot be defined on a proper sub-ramified covering of $\rho_p$ (so the ramification order $p$ is denoted $p(\varphi)$);
\item
$(R_\varphi,\nabla)$ is a finite dimensional $\CC\lpr x_p\rpr$-vector space with a regular connection $\nabla$;
\item
$\El(\varphi,(R_\varphi,\nabla))$ is the push-forward (in the sense of vector spaces with meromorphic connection) of the twisted module $(R_\varphi,\nabla+\rd\varphi)$ by $\rho_{p(\varphi)}$.
\end{itemize}
The following is known as the Levelt-Turrittin theorem (\cf\cite[p\ptbl51]{Malgrange91}, see also \cite{Bibi07a}).
\begin{itemize}
\item
Each $\wh M_{x_i}$ decomposes as $\wh M_{x_i}^\reg\oplus\wh M_{x_i}^\irr$.
\item
Each $\wh M_{x_i}^\irr$ is a finite dimensional $\CC\lpr x\rpr$-vector space with connection, which decomposes as a direct sum of elementary modules $\El(\varphi,(R_\varphi,\nabla))$ where $\varphi$ and $(R_\varphi,\nabla)$ are as above; moreover, the decomposition is unique if we impose that, for $\varphi\neq\eta$ with $p(\varphi)=p(\eta)=p$ such that $R_\varphi,R_\eta\neq0$, there exists no $\zeta$ with $\zeta^p=1$ and $\eta(x_p)=\varphi(\zeta x_p)$. The \emph{slope} of $\El(\varphi,(R_\varphi,\nabla))$ is $q(\varphi)/p(\varphi)$. The set of $\varphi\neq0$ such that $R_\varphi\neq0$ is denoted by $\Phi_i\setminus\{0\}$.
\end{itemize}

On the other hand, a $\CC\lpr x\rpr$-vector space with regular connection $(R,\nabla)$ is canonically equipped with a decreasing filtration $V^a$ (\resp $V^{>a}$) indexed by $a\in\RR$, where~$V^a$ is the free $\CC\lcr x\rcr$-module on which the residue of $\nabla$ has eigenvalues $\alpha$ with real part in $[a,a+1)$ (\resp $(a,a+\nobreak1]$). The $\alpha$-eigenspace of the residue on $V^a/V^{a+1}$ is denoted by $\psi_\lambda(R)$, with $\lambda=\exp(-\twopii\alpha)$, and the corresponding nilpotent part of the residue is denoted by $\rN$. Giving $(R,\nabla)$ is then equivalent to giving the finite family of $\CC$-vector spaces $\psi_\lambda(R)$ ($\lambda\in\CC^*$), called the \emph{moderate nearby cycle spaces}, equipped with a nilpotent endomorphism~$\rN$.

Then $\wh M_{x_i}^\irr$ is completely determined by the finite family $\big(\varphi,\psi_\lambda(R_\varphi,\nabla),\rN\big)$ (\hbox{$\varphi\in\Phi_i\setminus\{0\}$}, $\lambda\in\CC^*$), that we also denote\footnote{\label{footnote:1}The notation $\psi_{t-x_i,\lambda}^\varphi(M),\phi_{t-x_i,\lambda}^\varphi(M)$, where we indicate the function which vanishes at $x_i$, would be more appropriate and is the standard notation. Since the coordinate $t$ is fixed, we will use the shortcuts $\psi_{x_i,\lambda}^\varphi(M)$ and $\phi_{x_i,\lambda}^\varphi(M)$.} by $\big(\psi_{x_i,\lambda}^\varphi(M),\rN\big)$.

Considering the regular part $\wh M_{x_i}^\reg$, we now denote by $\big(\psi_{x_i,\lambda}^0(M),\rN\big)$ (\ie $\varphi=0$ and $p(0)=1$, $q(0)=0$) the corresponding nearby cycle space with nilpotent operator, and we conclude that $\CC\lpr x\rpr\otimes_{\CC\lcr x\rcr}\wh M_{x_i}$ is completely determined by the finite family $\big(\varphi,\psi_{x_i,\lambda}^\varphi(M),\rN\big)$ (\hbox{$\varphi\in\Phi_i$}, $\lambda\in\CC^*$), that we simply denote by $\big(\psi_{x_i,\lambda}^\varphi(M),\rN\big)$.

As a consequence, $\wh M_{x_{r+1}}$ is completely determined by $\big(\psi_{x_{r+1},\lambda}^\varphi(M),\rN\big)_{\varphi\in\Phi_{r+1},\lambda\in\CC^*}$. At finite distance, $\big(\psi_{x_i,\lambda}^\varphi(M),\rN\big)_{\varphi\in\Phi_i,\lambda\in\CC^*}$ only determines $\CC\lpr x\rpr\otimes_{\CC\lcr x\rcr}\wh M_{x_i}$ \emph{a~priori}. However, if we assume that $M$ is a \emph{minimal (or intermediate) extension} at $x_i$, that is, has neither a sub nor a quotient module supported at $x_i$, then the family $\big(\psi_{x_i,\lambda}(M),\rN\big)_{\varphi\in\Phi_i,\lambda\in\CC^*}$ does determine $\wh M_{x_i}$. It is however useful to emphasize then the $\CC$-vector space with nilpotent endomorphism of \emph{moderate vanishing cycles} for the eigenvalue one of monodromy,
\[
\big(\phi_{x_i,1}(M),\rN\big):=\Big(\text{image}\big[\rN:\psi_{x_i,1}(M)\to\psi_{x_i,1}(M)\big],\text{induced }\rN\Big).
\]

In order to insure the minimality property, we will use the following criterion:

\begin{lemme}\label{lem:irredmini}
Assume that $M$ is an irreducible (or semi-simple) holonomic $\Clt$-module. Then $M$ is a minimal extension at each of its singular points.\qed
\end{lemme}

\begin{assumption}\label{ass:irrednotO}\mbox{}
\begin{enumerate}
\item\label{ass:irrednotO1}
In what follows, we always assume that $M$ is irreducible, not isomorphic to any $(\CC[t],\rd+c\,\rd t)$ ($c\in\CC$) and is not supported on a point.
\item\label{ass:irrednotO2}
We will sometimes assume that $M$ has a \emph{regular singularity} at $x_{r+1}$ and that the monodromy of $\DR M$ around $x_{r+1}$ is scalar and $\neq\id$, that is, of the form $\lambda_o\id$ with $\lambda_o\neq1$ (with the notation below, this means that ${}^\varphi\nu_{x_{r+1},\lambda,\ell}(M)=0$ unless $\varphi=0$, $\lambda=\lambda_o$ and $\ell=0$).
\end{enumerate}
\end{assumption}

\subsubsection*{Local numerical data attached to a holonomic module}
The ``monodromy filtration'' attached to the nilpotent endomorphism $\rN$ allows one to define, for each $\ell\in\NN$, the space of primitive vectors $\rP_\ell\psi_{x_i,\lambda}^\varphi(M)$ ($\varphi\in\Phi_i$), whose dimension is the number of Jordan blocks of size $\ell+1$ for $\rN$ acting on $\psi_{x_i,\lambda}^\varphi(M)$. We define the \emph{nearby cycle local data} ($i=1,\dots,r+1$, $\lambda\in\CC^*$) as
\begin{equation}\label{eq:nuvarphi}
\left\{
\begin{aligned}
{}^\varphi\nu_{x_i,\lambda,\ell}(M)&=p(\varphi)\cdot\dim\rP_\ell\psi_{x_i,\lambda}^\varphi(M)\quad(\ell\geq0),\\
{}^\varphi\nu_{x_i,\lambda,\prim}(M)&=\sum_{\ell\geq0}{}^\varphi\nu_{x_i,\lambda,\ell}(M),\\
{}^\varphi\nu_{x_i, \lambda}(M)&=\dim\psi_{x_i,\lambda}^\varphi(M)=\sum_{\ell\geq0}(\ell+1){}^\varphi\nu_{x_i,\lambda,\ell}(M),
\end{aligned}\right.
\end{equation}
and we have
\begin{equation}\label{eq:nuM}
\begin{split}
h(M)&=h_{x_i}(M)\\
&=\sum_{\varphi\in\Phi_i}\sum_\lambda{}^\varphi\nu_{x_i, \lambda}(M)\quad\text{(generic rank of $M$, independent of $i$)}.
\end{split}
\end{equation}
On the other hand, we define the \emph{vanishing cycle local data} by setting
\begin{equation}\label{eq:munu}
{}^\varphi\mu_{x_i,\lambda,\ell}(M)={}^\varphi\nu_{x_i,\lambda,\ell}(M),
\end{equation}
except if $i\neq r+1$, $\varphi=0$ and $\lambda=1$, for which we consider the values for $\im\rN$ and we set
\begin{equation}\label{eq:muone}
\left\{\begin{aligned}
{}^0\mu_{x_i,1,\ell}(M)&={}^0\nu_{x_i,1,\ell+1},\\
{}^0\mu_{x_i,1,\prim}(M)&=\sum_{\ell\geq0}{}^0\mu_{x_i,1,\ell}(M)={}^0\nu_{x_i,1,\prim}(M)-{}^0\nu_{x_i,1,0},\\
{}^0\mu_{x_i,1}(M)&=\sum_{\ell\geq0}(\ell+1)\,{}^0\mu_{x_i,1,\ell}(M)\\
&=\sum_{\ell\geq0}\ell\,{}^0\nu_{x_i,1,\ell}(M)={}^0\nu_{x_i,1}-{}^0\nu_{x_i,1,\prim}(M).
\end{aligned}\right.
\end{equation}
We then set, for $\varphi\in\Phi_i$,
\begin{equation}\label{eq:muxi}
{}^\varphi\mu_{x_i}(M)=\sum_\lambda{}^\varphi\mu_{x_i, \lambda}(M),\quad\mu_{x_i}(M)=\sum_{\varphi\in\Phi_i}{}^\varphi\mu_{x_i}(M).
\end{equation}

\begin{definition}[Local numerical data]\label{def:locnumdata}
Let $M$ be an irreducible holonomic $\cD(\Afu)$-module with singular points $x_1,\dots,x_r,x_{r+1}=\infty$. The local numerical data of~$M$ consist~of
\begin{enumerate}
\item
the generic rank $h(M)$,
\item
the nearby cycle local data ${}^\varphi\nu_{x_{r+1},\lambda,\ell}(M)$ ($\lambda\in\CC^*$, $\ell\in\NN$),
\item
the vanishing cycle local data ${}^\varphi\mu_{x_i,\lambda,\ell}(M)$ ($i=1,\dots,r$, $\lambda\in\CC^*$, $\ell\in\NN$).
\end{enumerate}
\end{definition}

\begin{remarque}\label{rem:missingnu}
The formula for the missing $^0\nu_{x_i,1,\ell}(M)$'s, $i=1,\dots,r$, is:
\[
{}^0\nu_{x_i,1,\ell}(M)=\begin{cases}
{}^0\mu_{x_i,1,\ell-1}(M)&\text{if }\ell\geq1,\\
h(M)-\mu_{x_i}(M)-{}^0\mu_{x_i,1,\prim}(M)&\text{if }\ell=0.
\end{cases}
\]

As we will see below, the $\nu$'s and $h$ behave well when tensoring with a Kummer module, while the $\mu$'s behave well by Fourier transform. The behaviour under middle convolution by a Kummer module will therefore be more complicated.
\end{remarque}

\subsection{A quick review of middle convolution for holonomic modules on the affine line}\label{subsec:quickreview}

We review the notions and results introduced by Katz \cite{Katz96}, in the frame of holonomic $\cD$-modules (\cf \cite{Arinkin08}).

Let $\Afu$ be the affine line and let $M,N$ be holonomic $\cD(\Afu)$-modules. The \emph{external product} $M\boxtimes N$ is a holonomic $\cD(\Afu\times\Afu)$-module. The \emph{(internal) tensor product} $M\otimes^{\bL} N$ is the pull-back $\delta^+(M\boxtimes N)$ of the external product by the diagonal embedding $\delta:\Afu\hto\Afu\times\Afu$. It is an object of $\catD^\rb(\cD(\Afu))$. If $N$ is $\cO(\Afu)$-flat, then $M\otimes^{\bL} N=M\otimes N$ is a holonomic $\cD(\Afu)$-module.

\subsubsection*{Convolutions}
Consider the map $s:\Afu_x\times\Afu_y\to\Afu_t$ (where the index is the name of the corresponding variable) defined by $t=s(x,y)=x+y$. The (usual) $*$-convolution $M\star_*N$ of $M$ and $N$ is the object $s_+(M\boxtimes N)$ of $\catD^\rb(\Clt)$. The $*$-convolution is associative. On the other hand, the $!$-convolution is defined as the adjoint by duality of the $*$-convolution: $M\star_!N=\bD(\bD M\star_*\bD N)$, where $\bD$ is the duality functor $\catD^{\rb,\op}_{\hol}(\cD(\Afu))\to \catD^\rb_{\hol}(\cD(\Afu))$. It is also expressed as $s_\dag(M\boxtimes N)$, if~$s_\dag:=\bD s_+\bD$ denotes the adjoint by duality of $s_+$. Similarly, $\star_!$ is associative.

Let us choose a projectivization $\wt s:X\to\Afu_t$ of $s$, and let $j:\Afu_x\times\Afu_y\hto X$ denote the open inclusion. Since $\wt s$ naturally commutes with duality, we have $\wt s_\dag=\wt s_+$ and $s_\dag=\wt s_+\circ j_\dag$. Since there is a natural morphism $j_\dag\to j_+$ in $\catD^\rb_\hol(\cD_X)$, we get a functorial morphism $s_\dag(M\boxtimes N)\to s_+(M\boxtimes N)$, that is, $M\star_!N\to M\star_*N$, in $\catD^\rb_\hol(\cD(\Afu))$.

\subsubsection*{Convolutions and Fourier transform}
Let $\tau$ be the variable which is Fourier dual to $t$. The Fourier transform (with kernel $e^{-t\tau}$) of a $\Clt$-module $M$ is denoted by $\Fou M$. It is equal to the $\CC$-vector space $M$ on which $\Cltau$ acts as follows: $\tau$ acts as $\partial_t$ and $\partial_\tau$ acts as $-t$. If $\wh p:\Afu_t\times\Afu_\tau\to\Afu_\tau$ denotes the projection, it is also expressed as $\wh p_+(M[\tau]e^{-t\tau})=H^0\wh p_+(M[\tau]e^{-t\tau})$.

Due to the character property of the Fourier kernel, we have
\begin{equation}\label{eq:Fou*}
\Fou(M\star_*N)=\Fou M\otimes^{\bL}\Fou N.
\end{equation}

Recall (\cf \cite[pp\ptbl86\,\&\,224]{Malgrange91}) that Fourier transform is compatible with duality up to a sign $\iota:\tau\mto-\tau$, that is, $\bD\Fou M\simeq\iota^+\Fou\bD M$. It follows that
\begin{equation}\label{eq:Fou!}
\Fou(M\star_!N)=\bD(\bD\Fou M\otimes^{\bL}\bD\Fou N).
\end{equation}

Denote by $\delta_0=\Clt/\Clt\cdot t$ the Dirac (at $0$) $\cD(\Afu)$-module, which satisfies $\Fou\delta_0=\CC[\tau]$. Then, clearly, $\delta_0$ is a unit both for $\star_*$ and $\star_!$.

\subsubsection*{Middle convolution with $\catP$}
We introduce the full subcategory $\catP$ of $\Mod_\hol(\Clt)$ consisting of holonomic $\Clt$-modules $N$ such that $\Fou N$ and $\Fou\bD N$ (equivalently, $\bD \Fou N$) are $\CC[\tau]$-flat. Clearly, $\catP$ is a full subcategory of $\Mod_\hol(\Clt)$ which is stable by duality.

From \eqref{eq:Fou*} and \eqref{eq:Fou!} it follows that, for $N$ in $\catP$ and any holonomic $M$ (\resp for $M$ in $\catP$), both $M\star_*N$ and $M\star_!N$ are holonomic $\cD(\Afu)$-modules (\resp belong to $\catP$). Clearly, $\delta_0$ belongs to $\catP$.

\begin{definition}
For $N$ in $\catP$ and $M$ holonomic, the middle convolution $M\star_\Mid N$ is defined as the image of $M\star_!N\to M\star_*N$ in $\Mod_\hol(\cD(\Afu))$.
\end{definition}

\begin{lemme}\label{lem:starexact}
For $N\in\catP$, $\star_*N$ and $\star_!N$ are exact functors on $\Mod_\hol(\Clt)$. On the other hand, $\star_\Mid N$ preserves injective morphisms as well as surjective morphisms.
\end{lemme}

\begin{proof}
The first assertion follows from the exactness of $\Fou$ and $\bD$ on $\Mod_\hol(\Clt)$, and the exactness of $\otimes\Fou N$ on $\Mod_\hol(\Cltau)$. The second one is then obvious.
\end{proof}

\subsubsection*{Middle convolution with Kummer modules}
For $\chi\in\CC^*\setminus\{1\}$, choose $\alpha\in\CC\setminus\nobreak\ZZ$ such that $\exp(-\twopii\alpha)=\chi$ and let $L_\chi$ (or $L_\chi(t)$ to make precise the variable~$t$) denote the Kummer $\Clt$-module $(\CC[t,t^{-1}],\rd+\alpha\,\rd t/t)=\Clt/\Clt\cdot\nobreak(t\partial_t-\nobreak\alpha)$. It does not depend on the choice of $\alpha$ up to isomorphism. The second presentation makes it clear that $\bD L_\chi=L_{\chi^{-1}}$ and $\Fou(L_\chi(t))=L_{\chi^{-1}}(\tau)$. It~will also be convenient to set $L_1=\delta_0$. From \eqref{eq:Fou*} and \eqref{eq:Fou!} we conclude that $L_\chi$ belongs to $\catP$ and, for any of the three~$\star$-products,
\begin{equation}\label{eq:LLchi}
L_\chi\star L_{\chi'}=L_{\chi\chi'}\quad\text{if }\chi\chi'\neq1.
\end{equation}
We also deduce that $L_\chi\star_\Mid L_{\chi^{-1}}=L_1$. If $M$ is a holonomic $\Clt$-module, we set
\begin{equation}\label{eq:MCchi}
\MC_\chi(M):=M\star_\Mid L_\chi.
\end{equation}

\begin{proposition}\label{prop:Mchimin}
If $\chi\neq1$, $\Fou\MC_\chi(M)$ is the minimal extension at the origin of $\Fou M\otimes L_{\chi^{-1}}$.
\end{proposition}

\begin{proof}
By construction, $\Fou\MC_\chi(M)\subset\Fou M\otimes L_{\chi^{-1}}$, hence it has no submodule supported at the origin. Since $\bD\Fou L_\chi=L_\chi$, $\bD\Fou M\otimes\bD\Fou L_\chi$ satisfies the same property, and thus its dual module has no quotient module supported at the origin. As a consequence, $\Fou\MC_\chi(M)$ does not have any quotient module supported at the origin. It remains therefore to show that $\Fou\MC_\chi(M)\otimes\CC[\tau,\tau^{-1}]=\Fou M\otimes L_{\chi^{-1}}$, \ie $\Fou\MC_\chi(M)$ and $\Fou M\otimes L_{\chi^{-1}}$ have the same generic rank. We will restrict to a non-empty Zariski open subset not containing the singularities of $\Fou M$ and $L_\chi(\tau)$. It is then a matter of showing that, on this open subset, $\bD(\bD\Fou M\otimes L_\chi(\tau))$ has the same rank as $\Fou M$ (or $\Fou M\otimes L_{\chi^{-1}}(\tau)$). On such an open set, the dual as a $\cD$-module coincides with the dual as a bundle with connection, so the assertion is clear.
\end{proof}

\begin{proposition}\label{prop:starmidcomm}
The middle convolution functor $\MC_\chi:\Mod_\hol(\Clt)\to\Mod_\hol(\Clt)$ satisfies $\MC_{\chi\chi'}\simeq\MC_{\chi'}\circ\MC_\chi$ if $\chi\chi'\neq1$ and $\MC_1=\id$. It~satisfies $\MC_{\chi^{-1}}\circ\MC_\chi=\id$ on non-constant irreducible holonomic modules. In particular, it sends non-constant irreducible holonomic modules to non-constant irreducible holonomic modules.
\end{proposition}

\begin{proof}
That $\MC_1=\id$ is clear by Fourier transform and can also be seen as follows: since $M\boxtimes L_1$ is nothing but the push-forward $i_+M$ by the inclusion $i:\Afu_x\times\{0\}\hto\Afu_x\times\Afu_y$, we have $s_!(M\boxtimes L_1)=s_+(M\boxtimes L_1)=M$ since $s\circ i=\id$; then the image of one into the other is also equal to~$M$.

We are reduced to proving an associativity property\footnote{We present a proof not relying on the Fourier transform for further use in the proof of Theorem \ref{th:MCchi}. Using Fourier transform, one can argue as follows if $\chi\chi'\neq1$. By Proposition \ref{prop:Mchimin}, it is enough to check the equality of the localization at the origin of the Fourier transforms of both terms, which is then straightforward.}. Indeed,
\begin{align*}
\MC_{\chi\chi'}(M)&=M\star_\Mid(L_\chi\star_\Mid L_{\chi'}),
\\
\MC_{\chi'}\circ\MC_\chi(M)&=(M\star_\Mid L_\chi)\star_\Mid L_{\chi'}.
\end{align*}

On the one hand, let us consider the following diagram:
\[
\xymatrix{
M\star_!L_\chi\star_!L_{\chi'}\ar@{->>}[r]\ar@{->>}[d]&(M\star_\Mid L_\chi)\star_!L_{\chi'}\ar@{->>}[d]\ar@{^{ (}->}[r]&(M\star_*L_\chi)\star_!L_{\chi'}\ar@{->>}[d]\\
(M\star_!L_\chi)\star_\Mid L_{\chi'}\ar@{->>}[r]\ar@{^{ (}->}[d]&(M\star_\Mid L_\chi)\star_\Mid L_{\chi'}\ar@{^{ (}->}[r]\ar@{^{ (}->}[d]&(M\star_*L_\chi)\star_\Mid L_{\chi'}\ar@{^{ (}->}[d]\\
(M\star_!L_\chi)\star_*L_{\chi'}\ar@{->>}[r]&(M\star_\Mid L_\chi)\star_*L_{\chi'}\ar@{^{ (}->}[r]&M\star_*L_\chi\star_*L_{\chi'}
}
\]
Lemma \ref{lem:starexact} shows that the horizontal arrows have the indicated surjectivity/injectivity property. The vertical arrows of the first line are onto and those of the second line are injective by definition of $\star_\Mid$. It follows that $(M\star_\Mid L_\chi)\star_\Mid L_{\chi'}$ is also identified with the image of $M\star_!L_\chi\star_!L_{\chi'}\to M\star_*L_\chi\star_*L_{\chi'}$.

On the other hand, if $\chi\chi'\neq1$, $M\star_\Mid(L_\chi\star_\Mid L_{\chi'})=M\star_\Mid(L_\chi\star L_{\chi'})$ (any~$\star$), so
\[
M\star_\Mid(L_\chi\star L_{\chi'})\\
=\text{image}[M\star_!(L_\chi\star_! L_{\chi'})\to M\star_*(L_\chi\star_* L_{\chi'})].
\]
If $\chi'=\chi^{-1}$, we consider the diagram
\[
\xymatrix{
M\star_!(L_\chi\star_!L_{\chi^{-1}})\ar[rr]\ar[d]&&M\star_*(L_\chi\star_*L_{\chi^{-1}})\\
M\star_!L_1\ar@{=}[r]&M\star_\Mid L_1\ar@{=}[r]&M\star_*L_1\ar[u]
}
\]
where in the lower line, all terms are nothing but $M$, and we are reduced to proving that the left vertical morphism is onto and the right one in injective if~$M$ is irreducible and non-constant. Since $\Fou(L_\chi\star_*L_{\chi^{-1}})=\CC[\tau,\tau^{-1}]$ and~$\Fou M$ and $\Fou\bD M$ have no $\tau$-torsion by our assumption, the morphism $M\to M\star_*L_\chi\star_*L_{\chi^{-1}}$ is injective and so is $\bD M\to\bD M\star_*L_\chi\star_*L_{\chi^{-1}}$, hence by duality $M\star_!L_\chi\star_!L_{\chi^{-1}}\to M$ is onto.
\end{proof}

Let us decompose $\Afu_x\times\Afu_y$ as $\Afu_x\times\Afu_t$ so that $s$ becomes the second projection~$p$, and let us choose $\wt s:\PP^1_x\times\Afu_t\to\Afu_t$ as being the second projection. Since $j:\nobreak\Afu_x\times\nobreak\Afu_t\hto\PP^1_x\times\Afu_t$ is affine, $j_\dag(M\boxtimes L_\chi),j_+(M\boxtimes L_\chi),j_{\dag+}(M\boxtimes L_\chi)$ are holonomic $\cD_{\PP^1_x\times\Afu_t}$-modules, as well as the kernel $K$ of $j_\dag(M\boxtimes L_\chi)\to j_{\dag+}(M\boxtimes L_\chi)$ and the cokernel~$C$ of $j_{\dag+}(M\boxtimes L_\chi)\to j_+(M\boxtimes L_\chi)$, and we denote the middle extension $j_{\dag+}(M\boxtimes L_\chi)$ by~$\ccM_\chi$. If $M$ is irreducible, $\ccM_\chi$ is irreducible since $M\boxtimes L_\chi$ is an irreducible $\cD_{\Afu_x\times\Afu_t}$-module.

\begin{proposition}\label{prop:MCchiconvolution}
We have $\MC_\chi(M)=\wt s_+\ccM_\chi$.
\end{proposition}

\begin{proof}
Since $L_\chi$ belongs to $\catP$, we have $s_\dag(M\boxtimes L_\chi)=H^0s_\dag(M\boxtimes L_\chi)$ and $s_+(M\boxtimes\nobreak L_\chi)=H^0s_+(M\boxtimes L_\chi)$, so that
\begin{align*}
\MC_\chi(M)&=\text{image}\Big[H^0s_\dag(M\boxtimes L_\chi)\to H^0s_+(M\boxtimes L_\chi)\Big]\\
&=\text{image}\Big[H^0\wt s_+j_\dag(M\boxtimes L_\chi)\to H^0\wt s_+j_+(M\boxtimes L_\chi)\Big].
\end{align*}
On the other hand, since $\wt s$ is the identity when restricted to $\{x_{r+1}\}\times\Afu_t$, both $\wt s_+K$ and $\wt s_+C$ have cohomology in degree zero only. It follows that $H^0\wt s_+j_\dag(M\boxtimes L_\chi)\to H^0\wt s_+j_{\dag+}(M\boxtimes L_\chi)$ is onto and $H^0\wt s_+j_{\dag+}(M\boxtimes L_\chi)\to H^0\wt s_+j_+(M\boxtimes L_\chi)$ is injective, which proves the equality $\MC_\chi(M)=H^0\wt s_+\ccM_\chi$. That $H^k\wt s_+\ccM_\chi=0$ for $k\neq0$ is proved similarly.
\end{proof}

\subsection{Behaviour of local data by various operations}
\subsubsection*{Twist by a rank-one local system}
Let us fix a singular point $x_i$ of $M$ with local coordinate $x$ as above, and let us focus on local data at $x_i$. For $\lambda_i\in\CC^*\setminus\{1\}$, let $\wh L_{x_i,\lambda_i}$ denote the $\CC\lpr x\rpr$-vector space $\CC\lpr x\rpr$ equipped with the connection $\rd+\nobreak\alpha_i\,\rd x/x$, where $\alpha_i$ is chosen so that $\exp(-\twopii\alpha_i)=\lambda_i$. A standard computation (see a more precise computation in \eqref{eq:FpVL}) shows that, for each $\varphi\in\Phi_i$ and each $\lambda\in\CC^*$, we have
\[
\big(\psi_{\lambda}^\varphi(\wh M_{x_i}\otimes\wh L_{x_i,\lambda_i}),\rN\big)\simeq\big(\psi_{\lambda/\lambda_i^{p(\varphi)}}^\varphi(\wh M_{x_i}),\rN\big).
\]
We thus get
\[
^\varphi\nu_{\lambda,\ell}((\wh M_{x_i}\otimes\wh L_{x_i,\lambda_i})_{\min})={}^\varphi\nu_{\lambda,\ell}(\wh M_{x_i}\otimes\wh L_{x_i,\lambda_i})={}^\varphi\nu_{\lambda/\lambda_i^{p(\varphi)},\ell}(\wh M_{x_i})
\]
The formulas \eqref{eq:muone} also allow one to compute the missing local data $\mu_\bbullet$ of the minimal extension at $x_i$ of $\wh M_{x_i}\otimes\wh L_{x_i,\lambda_i}$ from ${}^0\nu_{x_i,1/\lambda_i,\ell}(\wh M_{x_i})$:
\[
{}^0\mu_{1,\ell}((\wh M_{x_i}\otimes\wh L_{x_i,\lambda_i})_{\min})={}^0\nu_{1/\lambda_i,\ell+1}(\wh M_{x_i}).
\]
As a conclusion:
\begin{align}\label{eq:nutwist}
h((\wh M_{x_i}\otimes\wh L_{x_i,\lambda_i})_{\min})&=h(\wh M_{x_i}),\\\label{eq:nulambdatwist}
{}^\varphi\nu_{\lambda,\ell}(\wh M_{x_{r+1}}\otimes\wh L_{x_{r+1},\lambda_{r+1}})&={}^\varphi\nu_{\lambda/\lambda_{r+1}^{p(\varphi)},\ell}(\wh M_{x_{r+1}})\\\label{eq:mutwist}
{}^\varphi\mu_{\lambda,\ell}((\wh M_{x_i}\otimes\wh L_{x_i,\lambda_i})_{\min})&\qquad(i=1,\dots,r)\\\notag
&\hspace*{-3.5cm}=\begin{cases}
{}^\varphi\mu_{\lambda/\lambda_i^{p(\varphi)},\ell}(\wh M_{x_i})&\text{if }\varphi\neq0\text{ or }\lambda\neq1,\lambda_i,\\
{}^0\mu_{1/\lambda_i,\ell+1}(\wh M_{x_i})&\text{if }\varphi=0\text{ and }\lambda=1,\\
{}^0\mu_{1,\ell-1}(\wh M_{x_i})&\text{if }\varphi=0,\;\lambda=\lambda_i\text{ and }\ell\geq1,\\
h(\wh M_{x_i})-\mu(\wh M_{x_i})-{}^0\mu_{1,\prim}(\wh M_{x_i})&\text{if }\varphi=0,\;\lambda=\lambda_i\text{ and }\ell=0.
\end{cases}
\end{align}

A straightforward computation gives then ($i=1,\dots,r$)
\begin{equation}\label{eq:mutwistmoinsmu}
\mu((\wh M_{x_i}\otimes\wh L_{x_i,\lambda_i})_{\min})=h(\wh M_{x_i})-{}^0\mu_{1/\lambda_i,\prim}(\wh M_{x_i}).
\end{equation}

\subsubsection*{Behaviour of local data by Fourier transform}
The main tool will be here the formal stationary phase formula. Let $M$ be an irreducible holonomic $\Clt$-module and let $\Fou M$ be its Fourier transform, which is also irreducible. Both are therefore minimal extensions at their singular points at finite distance. We denote by $\bmy=\{y_1,\dots,y_s\}$ the set of singular points at finite distance of $\Fou M$ and by $y_{s+1}=\wh\infty$ the point at infinity (in the variable $\tau$), which is also considered as a singular point of~$\Fou M$. It will be convenient to denote by $i_o$ (\resp $j_o$) the index $i\in\{1,\dots,r\}$ (\resp $j\in\{1,\dots,s\}$), if it exists, such that $x_{i_o}=0$ (\resp $y_{j_o}=0$).

Recall (\cf \cite{Laumon87} in the $\ell$-adic setting, and \cite{B-E04,Garcia04} for the $\cD$-module setting) that, for each $j=1,\dots,s+1$, $\wh{\Fou M}_{y_j}$ is obtained by local Fourier transform according to the formulas
\begin{align*}
\wh{\Fou M}_{y_j}&=\ccF^{(\infty,y_j)}(\wh M_\infty)\quad\text{for }j=1,\dots s,\\
\wh{\Fou M}_{\wh\infty}&=\bigoplus_{i=1}^{r+1}\ccF^{(x_i,\wh\infty)}(\wh M_{x_i}).
\end{align*}
The formal stationary phase formula (\cf \cite{Bibi07a,Fang07}) gives decompositions $\Phi_{r+1}(M)=\coprod_{j=1}^{s+1}\Phi_{r+1}^{(j)}(M)$ and $\Phi_{s+1}(\Fou M)=\coprod_{i=1}^{r+1}\Phi_{s+1}^{(i)}(\Fou M)$, and isomorphisms, that we simply denote by $\varphi\mto\wh\varphi$, and that we complement with the behaviour of the ramification order (recall that $q(\varphi)={}$order of the pole of~$\varphi$, $p(\varphi)={}$ramification order of~$\varphi$, \cf\S\ref{subsec:locnumdatahol}),
\begin{equation}\label{eq:PhiwhPhip}
\begin{aligned}
\Phi_{r+1}^{(j)}&\isom\Phi_j(\Fou M)\quad j=1,\dots,\wh{j_o},\dots,s,\\
\Phi_{r+1}^{(j_o)}&\isom\Phi_{j_o}(\Fou M),\quad p(\wh\varphi)=p(\varphi)-q(\varphi)\\
\Phi_i(M)&\isom\Phi_{s+1}^{(i)}(\Fou M),\quad p(\wh\varphi)=p(\varphi)+q(\varphi),\qquad i=1,\dots,r,\\
\Phi_{r+1}^{(s+1)}&\isom\Phi_{s+1}^{(r+1)}(\Fou M),\quad p(\wh\varphi)=q(\varphi)-p(\varphi).
\end{aligned}
\end{equation}
More precisely, $\Phi_{r+1}^{(s+1)}$ is the part with slope $>1$ in $\Phi_{r+1}$ ($\text{slope}(\varphi)=q(\varphi)/p(\varphi)$) and the image of $\Phi_{r+1}^{(s+1)}$ in $\Phi_{s+1}(\Fou M)$ is the part of slope~$>1$ in $\Phi_{s+1}(\Fou M)$, that is, $\Phi_{s+1}^{(r+1)}(\Fou M)$. Moreover, the union of the $\Phi_{s+1}^{(i)}(\Fou M)$'s for $i\in\{1,\dots,r\}\setminus\{i_o\}$ is the part of slope one, and $\Phi_{s+1}^{(i_o)}(\Fou M)$ (if nonempty) is the part of slope $<1$ in $\Phi_{s+1}(\Fou M)$. We have a similar description of $\Phi_{r+1}^{(j)}(M)$. We then obtain
\begin{align}
{}^{\wh\varphi_j}\mu_{y_j,\lambda,\ell}(\Fou M)&={}^{\varphi_j}\nu_{x_{r+1},\pm\lambda,\ell}(M),\quad\varphi_j\in\Phi_{r+1}^{(j)}(M)\; (j=1,\dots,s),\\
{}^{\wh\varphi_i}\nu_{y_{s+1},\lambda,\ell}(\Fou M)&={}^{\varphi_i}\mu_{x_i,\pm\lambda,\ell}(M), \\[-3pt]\notag
&\hspace*{2cm}\varphi_i\in\Phi_i(M)\; (i=1,\dots,r),\ \varphi_{r+1}\in\Phi_{r+1}^{(s+1)},
\end{align}
where $\pm\lambda$ is $(-1)^{q(\varphi)}\lambda$.

On the other hand, the rank of $\Fou M$ can be computed by using data at $y_{s+1}$. It is given by
\begin{equation}\label{eq:nuFouM}
h(\Fou M)=
\nu_{y_{s+1}}(\Fou M)
=\sum_{i=1}^r\mu_{x_i}(M)+{}^{(s+1)}\nu_{x_{r+1}}(M),
\end{equation}
which can be refined as
\begin{equation}
\begin{split}
{}^{(r+1)}\nu_{y_{s+1}}(\Fou M)&={}^{(s+1)}\nu_{x_{r+1}}(M),\\
{}^{(i)}\nu_{y_{s+1}}(\Fou M)&=\mu_{x_i}(M)\quad i=1,\dots,r.
\end{split}
\end{equation}
The previous formulas express the transformations of the local numerical data (Definition \ref{def:locnumdata}) of an irreducible holonomic $\Clt$-module by Fourier transform. Note that the formula for ${}^0\nu_{y_{j_o},1,0}(\Fou M)$, as given by the second line of Remark~\ref{rem:missingnu} for~$\Fou M$ can be given a global interpretation in terms of $M$ as follows. Let $\ccM$ denote the unique irreducible $\cD_{\PP^1}$-module whose restriction to $\Afu$ has global sections equal to~$M$.

\begin{proposition}
Under Assumption \ref{ass:irrednotO}\eqref{ass:irrednotO1}, we have
\[
{}^0\nu_{y_{j_o},1,0}(\Fou M)=\dim H^1(\PP^1,\DR\ccM).
\]
\end{proposition}

\begin{proof}
Let $\wt\ccM$ denote the localization of $\ccM$ at $x_{r+1}=\infty$, so that $H^k(\PP^1,\DR\wt\ccM)=H^k(\DR M)$. We have an exact sequence $0\to\ccM\to\wt\ccM\to\ccN\to0$, where~$\ccN$ is supported at $\infty$. Note that $H^k(\DR M)$ is the cohomology of the complex $M\To{\partial_t}M$, that we can write $\Fou M\To{\tau}\Fou M$. Standard results on the $V$-filtration of holonomic modules show that this complex is quasi-isomorphic to $\phi_{y_{j_o},1}^0(\Fou M)\To{\var}\psi_{y_{j_o},1}^0(\Fou M)$. Since $\Fou M$ is irreducible, it is a minimal extension at~$y_{j_o}$ (\cf Lemma \ref{lem:irredmini}), and therefore its $\var$ is injective and its $\can$ is onto. As a consequence, $H^k(\DR M)=0$ for $k\neq1$ and $\dim H^1(\DR M)=\dim\coker\var=\dim\coker\rN={}^0\nu_{y_{j_o},1,\prim}(\Fou M)$.

The exact sequence
\[
0\to\phi_{x_{r+1},1}^0(\ccM)\to\phi_{x_{r+1},1}^0(\wt\ccM)\simeq\psi_{x_{r+1},1}^0(\wt\ccM)\to\phi_{x_{r+1},1}^0(\ccN)\to0
\]
shows that $\dim\phi_{x_{r+1},1}^0(\ccN)={}^0\nu_{x_{r+1},1,\prim}(\wt\ccM)={}^0\nu_{x_{r+1},1,\prim}(M)$. Since $\ccN$ is supported at $x_{r+1}$, we obtain $H^k(\PP^1,\DR\ccN)=0$ for $k\neq1$ and $\dim H^1(\PP^1,\DR\ccN)={}^0\nu_{x_{r+1},1,\prim}(M)$. Moreover, $H^2(\PP^1,\DR\ccM)\simeq H^0(\PP^1,\DR\bD\ccM)=0$ since $\ccM$ (hence~$\bD\ccM$) satisfies \ref{ass:irrednotO}\eqref{ass:irrednotO1}. The exact sequence
\[
0\to H^1(\PP^1,\DR\ccM)\to H^1(\PP^1,\DR\wt\ccM)\to H^1(\PP^1,\DR\ccN)\to0,
\]
together with the previous considerations, gives
\begin{align*}
\dim H^1(\PP^1,\DR\ccM)&=\dim H^1(\DR M)-{}^0\nu_{x_{r+1},1,\prim}(M)\\
&={}^0\nu_{y_{j_o},1,\prim}(\Fou M)-{}^0\mu_{y_{j_o},1,\prim}(\Fou M)\\
&={}^0\nu_{y_{j_o},1,0}(\Fou M).\qedhere
\end{align*}
\end{proof}

\subsubsection*{Behaviour of the local numerical data by $\MC_\chi$}
Putting the previous results together, we obtain:

\begin{proposition}\label{prop:numdataMCchi}
If $M$ satisfies \ref{ass:irrednotO}\eqref{ass:irrednotO1} and $\chi\neq1$, we have
\begin{align*}
{}^\varphi\mu_{x_i,\lambda,\ell}(\MC_\chi(M))&={}^\varphi\mu_{x_i,\lambda/\chi^{p(\varphi)+q(\varphi)},\ell}(M)\\[-2pt]
&\hspace*{1cm}\forall i=1,\dots,r,\;\forall\varphi\in\Phi_i(M),\;\forall\lambda\in\CC^*,\;\forall\ell\geq0,\\
{}^\varphi\nu_{x_{r+1},\lambda,\ell}(\MC_\chi(M))
&=\begin{cases}
{}^\varphi\nu_{x_{r+1},\lambda/\chi^{q(\varphi)-p(\varphi)},\ell}(M)&\text{if }\varphi\in\Phi_{r+1}^{(s+1)},\\
{}^\varphi\nu_{x_{r+1},\lambda,\ell}(M)&\text{if }\left\{\begin{array}{l}\varphi\in\Phi_{r+1}^{(j)},\\j\in\{1,\dots,s\}\setminus\{j_o\},\end{array}\right.\\
{}^\varphi\nu_{x_{r+1},\lambda\chi^{p(\varphi)-q(\varphi)},\ell}(M)&\text{if }\varphi\in\Phi_{r+1}^{(j_o)}\setminus\{0\},\\
{}^0\nu_{x_{r+1},\lambda\chi,\ell}(M)&\text{if }\varphi=0,\;\lambda\neq1,\chi^{-1},\\
{}^0\nu_{x_{r+1},\chi,\ell+1}(M)&\text{if }\varphi=0,\;\lambda=1,\\
{}^0\nu_{x_{r+1},1,\ell-1}(M)&\text{if }\varphi=0,\;\lambda=\chi^{-1},\;\ell\geq1,\\
\dim H^1(\PP^1,\DR\ccM)
&\text{if }\varphi=0,\;\lambda=\chi^{-1},\;\ell=0,
\end{cases}\\[2pt]
h\big(\MC_\chi(M)\big)&=h(M)+\dim H^1(\PP^1,\DR\ccM)\\[-2pt]
&\hspace*{1.32cm}+{}^0\nu_{x_{r+1},1,\prim}(M)-{}^0\nu_{x_{r+1},\chi,\prim}(M).
\end{align*}
\end{proposition}

\begin{proof}
Since the first two equalities are straightforward, let us indicate the proof of the last one. By applying \eqref{eq:nuFouM} in the reverse direction and the first two equalities, we get
\begin{align*}
h(\MC_\chi(M))&-h(M)=\nu_{x_{r+1}}(\MC_\chi(M))-\nu_{x_{r+1}}(M)\\
&=\sum_{j=1}^s\big[\mu_{y_j}((\Fou M\otimes L_{\chi^{-1}})_{\min})-\mu_{y_j}(\Fou M)\big]\\[-9pt]
&\hspace*{3.5cm}+{}^{(r+1)}\nu_{y_{s+1}}(\Fou M\otimes L_{\chi^{-1}})-{}^{(r+1)}\nu_{y_{s+1}}(\Fou M)\\
&=\mu_{y_{j_o}}((\Fou M\otimes L_{\chi^{-1}})_{\min})-\mu_{y_{j_o}}(\Fou M)\\
&={}^0\nu_{y_{j_o},1,0}(\Fou M)+{}^0\mu_{y_{j_o},1,\prim}(\Fou M)-{}^0\mu_{y_{j_o},\chi,\prim}(\Fou M)\quad\text{by \eqref{eq:mutwistmoinsmu}}\\
&=\dim H^1(\PP^1,\DR\ccM)+{}^0\nu_{x_{r+1},1,\prim}(M)-{}^0\nu_{x_{r+1},\chi,\prim}(M).\qedhere
\end{align*}
\end{proof}

\subsection{The case of a scalar monodromy at infinity}\label{subsec:localdatareg}
Together with Assumption \ref{ass:irrednotO}\eqref{ass:irrednotO1}, we moreover assume \ref{ass:irrednotO}\eqref{ass:irrednotO2} and we choose $\chi=\lambda_o$.
\begin{corollaire}\label{cor:localdatareg}
Under these assumptions, $\MC_\chi(M)$ also fulfills Assumptions \ref{ass:irrednotO}\eqref{ass:irrednotO1} and \eqref{ass:irrednotO2}, with scalar monodromy at infinity equal to $\lambda_o^{-1}\id$. Moreover,
\begin{starequation}
h\big(\MC_\chi(M)\big)=\dim H^1(\PP^1,\DR\ccM)=\dim H^1(\DR M).
\end{starequation}\end{corollaire}

\begin{proof}
The first point directly follows from the formulas in Proposition \ref{prop:numdataMCchi}. Since $\nu_{x_{r+1},1}(M)=0$ and $\nu_{x_{r+1},\chi,\prim}(M)=\nu_{x_{r+1},\chi}(M)=\nu_{x_{r+1}}(M)=h(M)$, we also get from this proposition:
\[
h\big(\MC_\chi(M)\big)=\dim H^1(\PP^1,\DR\ccM)=\dim H^1(\DR M),
\]
where the last equality follows from the equality $\ccM=\wt\ccM$, by our supplementary assumption and the choice of $\chi$.
\end{proof}

We get a topological expression of $h\big(\MC_\chi(M)\big)$ in terms of the perverse sheaf $\DR^\an M$:
\begin{equation}\label{eq:hMCchi}
h\big(\MC_\chi(M)\big)=\sum_{i=1}^r\mu_{x_i}\DR^\an M-h(M),
\end{equation}
where $\mu_{x_i}\DR^\an M=\mu_{x_i}M$ if $M$ has a regular singularity at $x_i$, but both numbers are distinct otherwise:
\[
\mu_{x_i}\DR^\an M=\irr_{x_i}(M)+h(M)-{}^0\nu_{1,\prim}(\wh M_{x_i})=\irr_{x_i}(M)+\mu_{x_i}(M),
\]
where the last equality follows from \eqref{eq:muone}, and where $\irr_{x_i}(M):=\irr(\wh M_{x_i})=\irr({}^\irr\wh M_{x_i})$ denotes the irregularity number of $M$ at $x_i$:
\[
\irr({}^\irr\wh M_{x_i}):=\sum_{\varphi\in\Phi_i(M)\setminus\{0\}}\text{slope}(\varphi)\cdot{}^\varphi\nu_{x_i}(M).
\]
We can thus also write
\[
\mu_{x_i}\DR^\an M=\sum_{\varphi\in\Phi_i(M)}(\text{slope}(\varphi)+1)\cdot{}^\varphi\mu_{x_i}(M).
\]

\begin{remarque}
In such a case, $\ccM_\chi$ is smooth along $\{x_{r+1}\}\times\Afu_t$, because the monodromy of $M\boxtimes L_\chi$ around this divisor is equal to the identity.
\end{remarque}

\begin{remarque}[Degree of $V^0$]
We keep the same setting as above and let us assume that $M$, hence also $\MC_\chi(M)$, has only regular singularities. Then the locally defined $\CC\lcr x\rcr$-modules $V^a$ at each singularity (\cf\S\ref{subsec:locnumdatahol}) are the formalization of some $\cO_{\PP^1}$-modules that we simply denote by $V^a(M)$ and $V^a(\MC_\chi(M))$, which are locally free if $a>-1$. We will be mostly interested by $V^0$ and its degree, which is nonpositive, according to the definition of $V^0$ and the residue formula. The residue formula, together with Proposition \ref{prop:numdataMCchi} (in the regular case), and with Assumptions \ref{ass:irrednotO}\eqref{ass:irrednotO1} and \eqref{ass:irrednotO2}, gives, for $\chi=\lambda_o\neq1$ and $\alpha_o\in(0,1)$ such that $\exp-\twopii\alpha_o=\lambda_o$:
\[
\deg V^0\big(\MC_\chi(M)\big)=\deg V^0(M)+h(M)-\sum_{i=1}^r\sum_{\alpha\in[0,1-\alpha_o)}\mu_{x_i,\exp(-\twopii\alpha)}(M).
\]
\end{remarque}

\subsection{A quick review of Katz' existence algorithm for rigid local systems}\label{subsec:Katzalg}
Let $\cV$ be an irreducible local system on $\Afu\setminus\bmx$, let $(V,\nabla)$ be the associated holomorphic bundle with connection. Its Deligne meromorphic extension on $\PP^1$ is a holonomic $\cD_{\PP^1}$-module. Let $\wt\ccM$ denote its minimal extension at each singularity at finite distance, and set $M=\Gamma(\PP^1,\wt\ccM)$. This is an irreducible holonomic $\cD(\Afu)$\nobreakdash-module with \emph{regular singularities} at $\bmx$, and any such module is obtained in this way.\footnote{We use the Zariski topology when working with $\cD$-modules, while we use the analytic topology when working with holomorphic bundles, local systems or perverse sheaves. For vector bundles on $\PP^1$ or holonomic $\cD_{\PP^1}$-modules, we implicitly use a GAGA argument to compare both kinds of objects. Although we should distinguish both by an exponent ``alg'' or ``an'', we will let this to the reader.}

Let $j:\Afu\setminus\bmx\hto\PP^1$ denote the open inclusion. We say (\cf\cite{Katz96}) that $\cV$ is \emph{cohomologically rigid} if its \emph{index of rigidity} $\chi(\PP^1,j_*\cEnd(\cV))$ is equal to $2$. Recall (\cf\cite[\S1.1]{Katz96}) that this index is computed as
\[
\chi(\PP^1,j_*\cEnd(\cV))=(1-r)(\rk\cV)^2+\sum_{i=1}^{r+1}\dim C(A_i),
\]
where $C(A_i)$ is the centralizer of the local monodromy $A_i$ at $x_i$. With the notation \eqref{eq:nuvarphi}, we have
\begin{align*}
\dim C(A_i)&=\sum_{\lambda\in\CC^*}\sum_{k,\ell\geq0}\nu_{x_i,\lambda,\ell}\nu_{x_i,\lambda,k}\min\{k+1,\ell+1\}\\
&\leq\sum_{\lambda\in\CC^*}\sum_{k,\ell\geq0}\nu_{x_i,\lambda,k}\nu_{x_i,\lambda,\ell}(\ell+1)
=\sum_\lambda\nu_{x_i,\lambda,\prim}\cdot\nu_{x_i,\lambda}.
\end{align*}

It follows from \cite[Th\ptbl3.3.3]{Katz96} (in the $\ell$-adic setting) and from \cite[Th\ptbl4.3]{B-E04} (in the present complex setting) that if $M$ (that is, $\cV$) is rigid, then $\MC_\chi(M)$ is rigid for any nontrivial $\chi$. Moreover, as obviously follows from the formula above, if $\cL_{\lambda_1,\dots,\lambda_r}$ denotes the rank-one local system on $\Afu\setminus\bmx$ having monodromy $\lambda_i$ around $x_i$ (and hence $\lambda_{r+1}=(\lambda_1\cdots\lambda_r)^{-1}$ around $x_{r+1}$), then $\cV$ is rigid if and only if $\cV\otimes\cL$ is~so.

A celebrated theorem of Katz asserts that $\cV$ is rigid if and only if it can be brought, after an initial homography, by a successive application of tensor products by suitable rank-one local systems and middle convolutions by suitable Kummer sheaves, to a rank-one local system. We will quickly review this algorithm.\footnote{Since we will be mainly interested in the case with regular singularities, we will not review the irregular analogue, due to Deligne on the one hand, and Arinkin \cite{Arinkin08} on the other hand.}

We consider the category consisting of local systems on $\Afu\setminus\bmx$ satisfying Assumption \ref{ass:irrednotO}\eqref{ass:irrednotO2} (\ie the associated regular holonomic $\cD(\Afu)$-module $M$ satisfies this assumption). Given such a local system $\cV$, we will only consider tensor products by rank-one local systems and middle convolutions by Kummer sheaves which preserve the property of being in this category. We will call these operations (or rank-one local systems) ``allowed''. Proposition \ref{prop:numdataMCchi} shows that, if $\cV$ satisfies Assumption \ref{ass:irrednotO}\eqref{ass:irrednotO2}, the only allowed $\MC_\chi$ is for $\chi=\lambda_o$ (where $\lambda_o\id$ is the local monodromy of $\cV$ at $x_{r+1}$).

Given any local system $\cV$ on $\PP^1\setminus D$, for some finite set $D$, then up to adding a fake singular point, to tensoring by a rank-one local system and to pull-back by a homography, we can assume that $\cV$ belongs to the category above. For the Katz algorithm, we can therefore start from a rigid irreducible local system $\cV$ whose monodromy at $x_{r+1}$ is $\lambda_o\id$ with $\lambda_o\neq1$. In such a case, we have
\[
2=\chi(\PP^1,j_*\cEnd(\cV))=(2-r)(\rk\cV)^2+\sum_{i=1}^r\dim C(A_i).
\]

The main step of the algorithm consists therefore in the following lemma.

\begin{lemme}\label{lem:Lallowed}
Let $\cV$ be such a local system of rank $\geq2$. For each $i=1,\dots,r$, let $\lambda_i\in\CC^*$ be such that $\nu_{x_i,\lambda_i,\prim}=\max_{\lambda\in\CC^*}\nu_{x_i,\lambda,\prim}$ and let $\cL$ be the rank-one local system with monodromy $1/\lambda_i$ at $x_i$. Then $\cL$ is allowed for $\cV$ and the rank of the unique allowed $\MC_\chi(\cL\otimes\cV)$ is $<\rk\cV$.
\end{lemme}

\begin{proof}
We have $\mu_{x_i}(\cL\otimes\cV)=(\rk\cV-\nu_{x_i,\lambda_i,\prim})$ and $\dim C(A_i)\leq \nu_{x_i,\lambda_i,\prim}\rk\cV$, hence
\[
(r-2)(\rk\cV)^2+2\leq\sum_{i=1}^r\nu_{x_i,\lambda_i,\prim}\rk\cV=r(\rk\cV)^2-\rk\cV\sum_{i=1}^r\mu_{x_i}(\cL\otimes\cV),
\]
by the rigidity assumption, and therefore $\sum_{i=1}^r\mu_{x_i}(\cL\otimes\cV)<2\rk\cV$. Once we know that $\cL$ is allowed, \eqref{eq:hMCchi} gives $\rk\MC_\chi(\cL\otimes\cV)=\sum_{i=1}^r\mu_{x_i}(\cL\otimes\cV)-\rk\cV<\rk\cV$.

If $\cL$ were not allowed, then the monodromy of $\cL\otimes\cV$ at $x_{r+1}$ would be the identity since that of $\cV$ is already scalar, and a formula similar to \eqref{eq:hMCchi} would give $\dim H^1(\PP^1,j_*(\cL\otimes\cV))=\sum_{i=1}^r\mu_{x_i}(\cL\otimes\cV)-2\rk\cV$, hence $\sum_{i=1}^r\mu_{x_i}(\cL\otimes\cV)\geq2\rk\cV$, in contradiction with the inequality given by rigidity.
\end{proof}

\section{Basics on Hodge theory}
\subsection{Variations of complex Hodge structure}\label{subsec:VPCHS}
Let $X$ a complex manifold. A~variation of polarized complex Hodge structure of weight $w\in\ZZ$ on $X$ is a $C^\infty$ vector bundle $H$ on $X$ equipped with
\begin{itemize}
\item
a grading $H=\bigoplus_{p\in\ZZ}H^p$,
\item
a flat connection $D$,
\item
and a $D$-flat sesquilinear pairing $k$,
\end{itemize}
such that the decomposition is $k$-orthogonal, the pairing $k$ is $(-1)^w$-Hermitian, the connection satisfies
\begin{align*}
D'H^p&\subset(H^p\oplus H^{p-1})\otimes\cA_X^{1,0}\\
D''H^p&\subset(H^p\oplus H^{p+1})\otimes\cA_X^{0,1},
\end{align*}
and the Hermitian pairing $h$ defined by the properties that the decomposition is $h$\nobreakdash-orthogonal and $h_{|H^p}=i^{-w}(-1)^pk_{|H^p}$, is positive definite.

As a consequence, the filtration $F^pH:=\bigoplus_{q\geq p}H^q$ satisfies $D''F^pH\subset F^pH\otimes\nobreak\cA_X^{0,1}$ and $D'F^pH\subset F^{p-1}H\otimes\cA_X^{1,0}$. The holomorphic bundle $V\defin\ker D''$ is equipped with a flat holomorphic connection $\nabla\defin D'_{\ker D''}$ and a filtration $F^pV=V\cap F^pH$ by holomorphic sub-bundles (\ie such that each $\gr^p_FV$ is a holomorphic bundle) which satisfies $\nabla F^pV\subset F^{p-1}V\otimes\Omega^1_X$ for all $p$. We will also denote by $(V,F^\cbbullet V,\nabla,k)$ such a variation, since these data together with the corresponding conditions allows one to recover the data $(H,\bigoplus H^\cbbullet,D,k)$.

If $(H,\bigoplus H^\cbbullet,D,k)$ is a variation of polarized complex Hodge structure of weight~$w$, then $(H,\bigoplus H^\cbbullet,D,i^{-w}k)$ has weight $0$. In this way, one can reduce to weight $0$. If we do not care much on the precise polarization, which will be the case below, we can assume that the weight is zero. We will therefore not mention the weight by considering variations of \emph{polarizable} complex Hodge structure.

More generally, given a polarizable complex Hodge structure $(H_o,\bigoplus H_o^\cbbullet,k_o)$, the tensor product $H\otimes H_o$ is naturally equipped with the structure of a variation of polarizable complex Hodge structure.\enlargethispage{\baselineskip}%

\begin{definition}[The local invariant $h^p$]
Given a filtered holomorphic bundle $(V,F^\cbbullet V)$ on a \emph{connected} complex manifold $X$, we will set $h^p(V)=h^p(V,F^\cbbullet V)=\rk\gr^p_FV$.
\end{definition}

For a variation of (polarizable) complex Hodge structure, we thus have
\[
h^p(V)=\rk H^p.
\]

\subsection{Local Hodge theory at a singular point and local invariants}\label{subsec:localHodgetheory}
Let $\Delta$ be a disc with coordinate $t$ and let $(V,F^\cbbullet V,\nabla,k)$ be a variation of polarizable complex Hodge structure on the punctured disc $\Delta^*$. Let $j:\Delta^*\hto\Delta$ be the open inclusion.

\subsubsection*{Extension across the origin}\label{subsubsec:extorigin}
We have the various extensions $V^a\subset V^{-\infty}\subset j_*V$, where $V^{-\infty}$ is Deligne's meromorphic extension and $V^a$ (\resp $V^{>a}$) ($a\in\RR$) is the free $\cO_\Delta$-module on which the residue of $\nabla$ has eigenvalues in $[a,a+1)$ (\resp $(a,a+\nobreak1]$). Here, the property that the eigenvalues $\lambda=\exp(-\twopii a)$ of the monodromy have absolute value one follows from the similar property for variations of real Hodge structure (\cf\cite[\S11]{Zucker79}) and the standard trick of making a real variation from a complex variation by adding its complex conjugate (\cf\S\ref{subsec:complexHodgemodule}).

For each $a\in\RR$, the Hodge bundle $F^pV^a$ is defined as $j_*F^pV\cap V^a$. We have $t^m V^a=V^{a+m}$ in $V^{-\infty}$ for each $m\in\ZZ$. Since clearly $t^mj_*F^pV=j_*F^pV$ for each $m\in\ZZ$, we conclude that for each $a\in\RR$ and $m\in\ZZ$, $t^mF^pV^a=F^pV^{a+m}$.

We now denote by $M=V^{-\infty}_{\min}\subset V^{-\infty}$ the $\cD_\Delta$-submodule generated by $V^{>-1}$: $M=V^{>-1}+\nabla_{\partial_t}V^{>-1}+\cdots\subset V^{-\infty}$. Recall that, if $j:\Delta^*\hto\Delta$ denotes the inclusion and $\ccV=V^\nabla$ is the locally constant sheaf of horizontal sections, we have $\DR V^{-\infty}=\bR j_*\ccV$ while $\DR M=j_*\ccV$. We have a $V$-filtration of $M$:
\[
V^aM=V^a\text{ for $a>-1$},\quad V^{-1}M=\nabla_{\partial_t}V^0+V^{>-1},\text{ etc.}
\]
The $\cD_\Delta$-module $M$ is equipped with a filtration:
\begin{equation}\label{eq:goodfiltr}
F^pM=F^pV^{>-1}+\nabla_{\partial_t}F^{p+1}V^{>-1}+\cdots+\nabla_{\partial_t}^kF^{p+k}V^{>-1}+\cdots\subset F^pV^{-\infty}.
\end{equation}

That $F^pV^a$ is a vector bundle follows from the Nilpotent orbit theorem \cite[(4.9)]{Schmid73}, and one deduces that $F^\cbbullet M$ is a good filtration of $M$ (\cf also \cite[Prop\ptbl3.10]{Bibi05}). Moreover, it follows from Theorem \ref{th:Schmid} below (by computing the dimension of the fibre at the origin) that each $\gr^p_FV^a$ is a vector bundle.

\subsubsection*{Nearby cycles}
The nearby cycles at the origin are defined as follows: for $a\in(-1,0]$ and $\lambda=\exp(-\twopii a)$,
\[
\psi_\lambda(M)=\psi_\lambda(V^{-\infty})=\gr^a_V=V^a/V^{>a}.
\]
It is equipped with the nilpotent endomorphism $\rN=-(t\partial_t-a)$. The Hodge filtration
\[
F^p\psi_\lambda(M):=F^pV^a/F^pV^{>a}=F^p\psi_\lambda(V^{-\infty})
\]
satisfies $\rN F^p\psi_\lambda(M)\subset F^{p-1}\psi_\lambda(M)$.

\begin{theoreme}[Schmid \cite{Schmid73}]\label{th:Schmid}
If $(V,F^pV,\nabla)$ is part of a variation of polarizable complex Hodge structure, then $(\psi_\lambda(M),F^\cbbullet\psi_\lambda(M),\rN)$ is part of a nilpotent orbit. Moreover,
\begin{starequation}\label{eq:Schmid*}
\rk F^pV=\sum_{\lambda\in S^1}\dim F^p\psi_\lambda(M),
\end{starequation}
hence
\begin{starstarequation}\label{eq:Schmid**}
h^p(V)=\sum_{\lambda\in S^1}h^p\psi_\lambda(M).
\end{starstarequation}
\end{theoreme}

In the following, we will set $\nu^p_\lambda(V)=h^p\psi_\lambda(M)=h^p\psi_\lambda(V^{-\infty})$ for $\lambda\in\CC^*$ (in~fact $\lambda\in S^1$). Note that the associated graded nilpotent orbit (graded with respect to the monodromy filtration $W$ of $\rN$) has the same numbers $h^p(\gr^W\psi_\lambda(M))=h^p(\psi_\lambda(M))$. The Hodge filtration on $\gr^W\psi_\lambda(M)=\gr^W\psi_\lambda(V^{-\infty})$ splits with respect to the Lefschetz decomposition associated with~$\rN$. The primitive part $\rP_\ell\psi_\lambda(M)$, equipped with the filtration induced by that on $\gr_\ell^W\psi_\lambda(M)$ and a suitable polarization, is a polarizable complex Hodge structure (\cite{Schmid73}). We can thus define the numbers
\[
\nu^p_{\lambda,\ell}(M)=\nu^p_{\lambda,\ell}(V^{-\infty}):=h^p\big(\rP_\ell\psi_\lambda(M)\big)=\dim\gr^p_F\rP_\ell\psi_\lambda(M),
\]
which are a refinement of the numbers $\nu_{\lambda,\ell}(M)$ considered in \S\ref{subsec:locnumdatahol} (by taking $\varphi=\nobreak0$ there and forgetting the left exponent $0$, since $M$ has a regular singularity). Accor\-ding to the $F$-strictness of $\rN$ and the Lefschetz decomposition of $\gr^W\psi_\lambda(M)$, we have
\begin{equation}\label{eq:nuplambda}
\nu^p_\lambda(M)=\sum_{\ell\geq0}\sum_{k=0}^\ell\nu_{\lambda,\ell}^{p+k}(M),
\end{equation}
and we set
\begin{equation}\label{eq:nuplambdaprim}
\nu^p_{\lambda,\prim}(M)=\sum_{\ell\geq0}\nu_{\lambda,\ell}^p(M),\quad\nu^p_{\lambda,\coprim}(M)=\sum_{\ell\geq0}\nu_{\lambda,\ell}^{p+\ell}(M).
\end{equation}
We have
\[
\nu^p_\lambda(M)-\nu^{p-1}_\lambda(M)=\nu^p_{\lambda,\coprim}(M)-\nu^{p-1}_{\lambda,\prim}(M).
\]

\subsubsection*{Vanishing cycles}
For $\lambda\neq1$, we set $\phi_\lambda(M)=\psi_\lambda(M)$, together with the monodromy, with $\rN$ and with the filtration $F^p$. If we set $\mu^p_{\lambda,\ell}(M)=\dim\gr^p_F\rP_\ell\phi_\lambda(M)$, we thus have by definition, for $\lambda\neq1$,
\[
\mu_{\lambda,\ell}^p(M)=\nu_{\lambda,\ell}^p(M)\quad\forall p.
\]
Let us now focus on $\lambda=1$. We have by definition
\[
\phi_1(M)=\gr^{-1}_V(M).
\]
On the other hand, the filtration $F^\cbbullet\phi_1(M)$ is defined so that we have natural morphisms
\[
(\psi_1(M),\rN,F^\cbbullet)\To{\can}\hspace*{-18pt}\to\hspace*{2pt}(\phi_1(M),\rN,F^\cbbullet)\Hto{\var}(\psi_1(M),\rN,F^\cbbullet)(-1),
\]
where the Tate twist by $-1$ means a shift of the Hodge filtration by $-1$, so that the object $(\psi_1(M),\rN,F^\cbbullet)(-1)$ is also a mixed Hodge structure. Since $\can$ is strictly onto and $\var$ is injective, $(\phi_1(M),\rN,F^\cbbullet)$ is identified with $\im\rN$ together with the filtration $F^p\im\rN=\rN(F^p)$. We also have, by definition of the Hodge filtration on~$M$,
\[
F^p\phi_1(M)=\frac{F^{p-1}M\cap V^{-1}M}{F^{p-1}M\cap V^{>-1}M}.
\]
For $\ell\geq0$, we thus have
\[
F^p\rP_\ell\phi_1(M)=\rN(F^p\rP_{\ell+1}\psi_1(M)),
\]
and therefore
\[
\mu_{1,\ell}^p(M)=\nu_{1,\ell+1}^p(M).
\]

\begin{proposition}[{\cf\cite[Prop\ptbl2.1.3]{K-K87}}]\label{prop:vanishingcycles}
$(\phi_1(M),\rN,F^\cbbullet)$ is part of a nilpotent orbit. Moreover,
\begin{starequation}\label{eq:vanishingcycles*}
\mu_1^p(M)=\nu_1^p(M)-\nu_{1,\coprim}^p(M)=\nu_1^{p-1}(M)-\nu_{1,\prim}^{p-1}(M).
\end{starequation}\end{proposition}

Note that, using the Lefschetz decomposition for the graded pieces of the monodromy filtration of $(\phi_1(M),\rN)$, we also have
\begin{equation}\label{eq:muplambda}
\mu^p_\lambda(M)=\sum_{\ell\geq0}\sum_{k=0}^\ell\mu_{\lambda,\ell}^{p+k}.
\end{equation}
We will set, similarly to \eqref{eq:nuplambdaprim}:
\begin{equation}\label{eq:muplambdaprim}
\mu^p_{\lambda,\prim}(M)=\sum_{\ell\geq0}\mu_{\lambda,\ell}^p(M),\quad\mu^p_{\lambda,\coprim}(M)=\sum_{\ell\geq0}\mu_{\lambda,\ell}^{p+\ell}(M).
\end{equation}

\subsubsection*{Comparison between $M$ and $V^{-\infty}$}
Let us denote by $N$ the cokernel of $M\hto V^{-\infty}$ equipped with the filtration induced by that of $V^{-\infty}$. It is supported at the origin.

\begin{lemme}\label{lem:comparisonMV}
The sequence
\[
0\to(M,F^\cbbullet M)\to(V^{-\infty},F^\cbbullet V^{-\infty})\to(N,F^\cbbullet N)\to0
\]
is exact and strict. Moreover, $(N,F^\cbbullet N)$ is identified with the cokernel of the morphism $\var:\phi_1(M)\to\psi_1(M)(-1)$ of mixed Hodge structures. In particular,
\[
h^p(N)=0,\quad\mu_1^p(N)=\dim\gr^p_F(N)=\nu_{1,\prim}^{p-1}(M).
\]
\end{lemme}

\subsubsection*{Local Hodge numerical data}
The various numerical data that we already introduced are recovered from the following Hodge numerical data. We use the notation $(V^{-\infty},F^\cbbullet V^{-\infty})$ and $(M,F^\cbbullet M)$ as above.

\begin{definition}[Local Hodge data]
For $\lambda\!\in\!S^1$, $p\!\in\!\ZZ$ and $\ell\!\in\!\NN$, we set
\par(Local Hodge data for $V^{-\infty}$)
\begin{itemize}
\item
$\nu^p_{\lambda,\ell}(V^{-\infty})=\dim\gr^p_F\rP_\ell\psi_\lambda(V^{-\infty})$, where $\rP_\ell\psi_\lambda(V^{-\infty})$ denotes the primitive part of $\gr_\ell^W\psi_\lambda(V^{-\infty})$ (a polarizable Hodge structure of weight $w+\ell$),
\item
$\nu^p_{\lambda}(V^{-\infty})$ given by \eqref{eq:nuplambda},
\item
$\nu^p(V^{-\infty})=\sum_\lambda\nu^p_\lambda(V^{-\infty})$, so that $\nu^p(V^{-\infty})=h^p(V)$ by \eqref{eq:Schmid**},
\end{itemize}

(Local Hodge data for $M$)
\begin{itemize}
\item
$h^p(V)$,
\item
$\mu^p_{1,\ell}(M)=\dim\gr^p_F\rP_\ell\phi_1(M)$, where $\rP_\ell\phi_1(M)$ denotes the primitive part of $\gr_\ell^W\phi_1(M)$, and $\mu^p_{\lambda,\ell}(M)=\nu^p_{\lambda,\ell}(M)$ if $\lambda\neq1$,
\item
$\mu^p_{\lambda}(M)$ given by \eqref{eq:muplambda} and $\mu^p(M)=\sum_\lambda\mu^p_\lambda(M)$.
\end{itemize}
\end{definition}

\begin{remarque}\label{rem:missingnup}
When considering a minimal extension $M$, we only deal with the data $\mu^p_\bbullet$. Then the data $\nu^p_1$ are recovered from the data $\mu^p_\bbullet$ together with $h^p(V)$:
\[
\nu^p_{1,\ell}(M)=\begin{cases}
\mu^p_{1,\ell-1}(M)&\text{if }\ell\geq1,\\
h^p(V)-\mu^p(M)-\mu_{1,\coprim}^{p+1}(M)&\text{if }\ell=0.
\end{cases}
\]
\end{remarque}

\subsubsection*{Twist with a unitary rank-one local system}\label{subsubsec:twistHodge}
Let $\ccL$ be a nontrivial unitary rank-one local system on $\Delta^*$, determined by its monodromy $\lambda_o\in S^1\setminus\{1\}$, and let $(L,\nabla)$ be the associated bundle with connection. We simply denote by $L^\cbbullet$ the various Deligne extensions of $(L,\nabla)$. It will be easier to work with $L^0$. We set $\lambda_o=\exp(-\twopii\alpha_o)$ with $\alpha_o\in(0,1)$. Then, $L^0=L^{\alpha_o}$ and, for each $a\in\RR$,
\[
V^a\otimes L^0=(V^{-\infty}\otimes L^{-\infty})^{a+\alpha_o}\subset(V^{-\infty}\otimes L^{-\infty})^a.
\]
On the other hand, the Hodge bundles on $V\otimes L$ are $F^pV\otimes L$ so that, by Schmid's procedure, for each $a$,
\[
F^p(V^{-\infty}\otimes L^{-\infty})^a:=j_*(F^pV\otimes L)\cap (V^{-\infty}\otimes L^{-\infty})^a
\]
(intersection taken in $V^{-\infty}\otimes L^{-\infty}$) is a bundle, and we have a mixed Hodge structure by inducing $F^p(V^{-\infty}\otimes L^{-\infty})^a$ on $\gr^a_V(V^{-\infty}\otimes L^{-\infty})$. We claim that
\begin{equation}\label{eq:FpVL}
F^pV^a\otimes L^0=F^p(V^{-\infty}\otimes L^{-\infty})^{a+\alpha_o}.
\end{equation}
This amounts to showing
\[
(j_*F^p\cap V^a)\otimes L^0=j_*(F^pV\otimes L)\cap (V^{-\infty}\otimes L^{-\infty})^{a+\alpha_o},
\]
intersection taken in $V^{-\infty}\otimes L^{-\infty}$. The inclusion $\subset$ is clear, and the equality is shown by working with a local basis vector of $L^0$, which can also serve as a basis for $L$ and~$L^{-\infty}$.

We deduce that, in a way analogous to \eqref{eq:nutwist}, \eqref{eq:nulambdatwist} and \eqref{eq:mutwist},
\begin{align}\label{eq:nupVL}
h^p(V\otimes L)&=h^p(V),\\
\nu_{\lambda,\ell}^p(V^{-\infty}\otimes L^{-\infty})&=\nu_{\lambda/\lambda_o,\ell}^p(V^{-\infty})\\
\label{eq:muVL}
\mu_{\lambda,\ell}^p((V^{-\infty}\otimes L^{-\infty})_{\min})&=
\begin{cases}
\mu_{\lambda/\lambda_o,\ell}^p(M)&\text{if }\lambda\neq1,\lambda_o,\\
\mu_{1/\lambda_o,\ell+1}^p(M)&\text{if }\lambda=1,\\
\mu_{1,\ell-1}^p(M)&\text{if }\lambda=\lambda_o\text{ and }\ell\geq1,\\
\left.\begin{array}{l}
h^p(V)-\mu^p(M)\\
-\mu_{1,\coprim}^{p+1}(M)
\end{array}\right\}\hspace*{-2mm}
&\text{if }\lambda=\lambda_o\text{ and }\ell=0.
\end{cases}
\end{align}

\subsection{Hodge numerical data for a variation on $\Afu\setminus\bmx$}\label{subsec:HodgenumdataA1}
Let $\bmx=\{x_1,\dots,x_r\}$ ($r\geq1$) be a finite set of points in $\Afu$ and set $x_{r+1}=\infty\in\PP^1$. Let $(V,F^\cbbullet V,\nabla,k)$ be a variation of polarizable complex Hodge structure on $U=\Afu\setminus\bmx$. Together with the local Hodge numerical data at each $x_i$ ($i=1,\dots,r+1$) we consider the following global numerical data. We consider the Hodge bundles $\gr_F^pV^0$, whose ranks are the $h^p(V)$.

\begin{definition}[Global Hodge numerical data]
For each $p$, we set
\[
\delta^p(V)=\deg\gr_F^pV^0.
\]
\end{definition}

\subsubsection*{Global Hodge numerical data of a tensor product}
Let $(L,\nabla)$ be the holomorphic bundle with connection associated to a unitary rank-one local system on $U=\Afu\setminus\bmx$. We denote by $\alpha_i\in[0,1)$ the residue of the connection $(L^0,\nabla)$ at $x_i$ ($i=1,\dots,r+\nobreak1$), so that $\deg L^0=-\sum_{i=1}^{r+1}\alpha_i$. We now denote by $\nu^p_{x_i,\lambda}(V)$ etc.\ the local Hodge numerical data of $V$ at $x_i$ whenever $\lambda\neq1$, and for $\bma=(a_1,\dots,a_{r+1})$ we denote by $V^{\bma}$ the extension of $V$ equal to $V^{a_i}$ near $x_i$.

\begin{proposition}\label{prop:degreetensor}
With the notation as above, we have
\[
\delta^p(V\otimes L)=\delta^p(V)+h^p(V)\deg L^0+\sum_{i=1}^{r+1}\sum_{\substack{\alpha\in[1-\alpha_i,1)\\\lambda=\exp(-\twopii\alpha)}}\nu^p_{x_i,\lambda}(V).
\]
\end{proposition}

\begin{proof}
We deduce from \eqref{eq:FpVL} (at~each~$x_i$) that
\begin{align*}
\delta^p(V\otimes L)&=\deg\gr^p_F(V\otimes L)^0=\deg\big[(\gr^p_FV^{-\alphag})\otimes L^0\big]\quad\text{after \eqref{eq:FpVL}}\\
&=\deg\gr^p_FV^{-\alphag}+h^p(V)\deg L^0\\
&=\delta^p(V)+h^p(V)\deg L^0+\sum_{i=1}^{r+1}\sum_{\substack{\beta\in[-\alpha_i,0)\\\lambda=\exp(-\twopii\beta)}}\nu^p_{x_i,\lambda}(V).\qedhere
\end{align*}
\end{proof}

\subsubsection*{Hodge data on the de~Rham cohomology}
Let $\ccM$ denote the minimal extension of $V^{-\infty}$ at each of the singular points $x_i$ ($i=1,\dots,r+1$) and let $F^\cbbullet\ccM$ be the extended Hodge filtration as in \S\ref{subsubsec:extorigin}. The de~Rham complex $\DR\ccM=\{0\to\ccM\to\Omega^1_{\PP^1}\otimes\ccM\to0\}$ is filtered by
\[
F^p\DR\ccM=\{0\to F^p\ccM\to\Omega^1_{\PP^1}\otimes F^{p-1}\ccM\to0\},
\]
and this induces a filtration on the hypercohomology $\bH^\cbbullet(\PP^1,\DR\ccM)=H^\cbbullet(\PP^1,j_*\ccV)$, where $j:\Afu\setminus\bmx\hto\PP^1$ denotes the open inclusion. It is known (by applying the results of \cite{Zucker79} to the variation of polarized real Hodge structure obtained by taking the direct sum of the original complex variation with its complex conjugate, \cf\S\ref{subsec:complexHodgemodule}) that $F^\cbbullet H^m(\PP^1,j_*\ccV)$ underlies a polarizable complex Hodge structure. Note that, if $\ccV$ is irreducible and non constant (as in Assumption \ref{ass:irrednotO}\eqref{ass:irrednotO1}, in the present case with regular singularities), then $H^m(\PP^1,j_*\ccV)=0$ for $m\neq1$.

\begin{proposition}\label{prop:nupH1}
Assume that $(V,F^\cbbullet V,\nabla,k)$ has an underlying irreducible and non constant local system $\ccV$. Then
\numstareq%
\begin{multline*}\tag{\theequation}\label{eq:nupH1}
h^p(H^1(\PP^1,j_*\ccV))\\[-5pt]
=\delta^{p-1}(V)-\delta^p(V)-h^p(V)-\nu^{p-1}_{x_{r+1},1,\prim}(V)+\sum_{i=1}^r\big(\nu^{p-1}_{x_i,\neq1}(V)+\mu^p_{x_i,1}(M)\big).
\end{multline*}\let\theequation\oldtheequation%
\end{proposition}

\begin{remarque}
In the case of a unitary rank-one local system, this result is already obtained in \cite[(2.20.1-2)]{D-M86}.
\end{remarque}

\begin{proof}
It follows from \cite[Prop\ptbl3.10(iib)]{Bibi05} that the inclusion of the filtered subcomplex
\[
F^\cbbullet V^0\DR\ccM\defin\{0\to F^\cbbullet V^0\ccM\to\Omega^1_{\PP^1}\otimes F^{\cbbullet-1} V^{-1}\ccM\to0\}
\]
into the filtered de~Rham complex is a filtered quasi-isomorphism. By the degeneration at $E_1$ of the Hodge de~Rham spectral sequence, we conclude that
\begin{align*}
-h^p(H^1(\PP^1,j_*\ccV))&=\chi\big(\gr^p_F\bH^\cbbullet(\PP^1,\DR\ccM)\big)\quad\text{(irreducibility and non-triviality of $\ccV$)}\\
&=\chi\big(\bH^\cbbullet(\PP^1,\gr^p_F\DR\ccM)\big)\quad\text{(degeneration at $E_1$)}\\
&=\chi\big(\bH^\cbbullet(\PP^1,\gr^p_FV^0\DR\ccM)\big)\quad\text{(as indicated above)}\\
&=\chi\big(H^\cbbullet(\PP^1,\gr^p_FV^0\ccM)\big)-\chi\big(H^\cbbullet(\PP^1,\Omega^1_{\PP^1}\otimes\gr^{p-1}_FV^{-1}\ccM)\big)\\[-3pt]
&\hspace*{5.3cm}\text{($\cO$-linearity of the differential)}\\
&=\delta^p(V)+h^p(V)-\deg(\Omega^1_{\PP^1}\otimes\gr^{p-1}_FV^{-1}\ccM)-h^{p-1}(V)\\[-5pt]
&\hspace*{7.45cm}\text{(Riemann-Roch)}\\
&=\delta^p(V)+h^p(V)-\deg(\gr^{p-1}_FV^{-1}\ccM)+h^{p-1}(V).
\end{align*}
We now have
\begin{align*}
\deg(\gr^{p-1}_FV^{-1}\ccM)&=\delta^{p-1}(V)+\dim\gr^{p-1}_F(V^{-1}\ccM/V^0\ccM)\\
&=\delta^{p-1}(V)+\sum_{a\in[-1,0)}\dim\gr^{p-1}_F(\gr^a_V\ccM)\\
&=\delta^{p-1}(V)+\sum_{i=1}^{r+1}\big(\nu^{p-1}_{x_i,\neq1}(V)+\mu^p_{x_i,1}(M)\big),
\end{align*}
from which one deduces the formula in the proposition, according to \eqref{eq:Schmid**} and \eqref{eq:vanishingcycles*}.
\end{proof}

\begin{remarque}\label{rem:nupH1}
Let us keep the assumptions of Proposition \ref{prop:nupH1}. We also clearly have $H^0(\PP^1,j_*\ccV)=0$, and also $H^2(\PP^1,j_*\ccV)=H^0(\PP^1,j_*\ccV^\vee)=0$ since the dual local system $\ccV^\vee$ satisfies the same assumptions.

Let $\wt\ccM$ denote the $\cD_{\PP^1}$-module which is the minimal extension of $(V,\nabla)$ at all~$x_i$ ($i=1,\dots,r$) and the meromorphic extension $V^{-\infty}$ at $x_{r+1}=\infty$. The exact sequence of Lemma \ref{lem:comparisonMV} at $x_{r+1}$ induces an exact sequence of de~Rham complexes
\[
0\to(\DR M,F^\cbbullet\DR M)\to(\DR V^{-\infty},F^\cbbullet\DR V^{-\infty})\to(\DR N,F^\cbbullet\DR N)\to0
\]
where $F^p\DR M=\{0\to F^pM\To{\nabla}F^{p-1}M\otimes\Omega^1_\Delta\to0\}$ etc.\ and $N=N_{x_{r+1}}$ is defined as in Lemma \ref{lem:comparisonMV}. Applying the previous results, one can show that the spectral sequence of hypercohomology of the filtered complex $F^\cbbullet\DR\wt\ccM$ degenerates at $E_1$, and that we have an exact sequence of complex mixed Hodge structures
\[
0\to F^\cbbullet H^1(\PP^1,j_*\ccV)\to F^\cbbullet H^1(\Afu,j_*\ccV)\to F^\cbbullet H^1(\DR N)\to0.
\]
It follows that
\begin{starequation}\label{eq:nupH1*}
h^p(H^1(\Afu,j_*\ccV))=\delta^{p-1}(V)-\delta^p(V)-h^p(V)+\sum_{i=1}^r\big(\nu^{p-1}_{x_i,\neq1}(V)+\mu^p_{x_i,1}(M)\big).
\end{starequation}\end{remarque}

\begin{remarque}
Let us keep the assumptions of Proposition \ref{prop:nupH1} and let us moreover assume that the monodromy of $\ccV$ around $x_{r+1}=\infty$ does not admit $1$ as an eigenvalue (\eg it takes the form $\lambda_o\id$ for some $\lambda_o\in\CC^*\setminus\{1\}$, \cf \S\ref{subsec:localdatareg}). Then $N_{x_{r+1}}=0$ and $F^\cbbullet H^1(\Afu,j_*\ccV)=F^\cbbullet H^1(\PP^1,j_*\ccV)$ is also a pure complex Hodge structure.
\end{remarque}

\subsection{Existence of a variation of polarized complex Hodge structure on a local system}

We keep the notation as in \S\ref{subsec:HodgenumdataA1}. Given a local system $\ccV$ on $U=\Afu\setminus\bmx$, necessary conditions on this local system to underlie a variation of polarizable complex Hodge structure are:
\begin{enumerate}
\item\label{enum:condsemisimple}
the local system is semi-simple, \ie direct sum of irreducible local systems,
\item\label{enum:condnormone}
for each $i=1,\dots,r+1$, the eigenvalues of the local monodromy at $x_i$ have absolute value equal to one.
\end{enumerate}

We now consider the question of whether these conditions are sufficient.

\begin{theoreme}[{\cite[Cor\ptbl8.1]{Simpson90}}]\label{th:Simpson}
Assume that the local system $\ccV$ on $U$ is physically rigid (\cf\cite[(1.0.3)]{Katz96}) and semi-simple. Then $\ccV$ underlies a variation of polarizable complex Hodge structure if and only if Condition \eqref{enum:condnormone} above is fulfilled.\qed
\end{theoreme}

Since physical rigidity is best understood when $\ccV$ is irreducible (\cf\cite[Th\ptbl1.1.2]{Katz96}), it is simpler to work with irreducible local systems. Reducing to irreducible local systems does not cause trouble when considering variations of Hodge structure, according to the following result.

Let $(V,F^\cbbullet V,\nabla,k)$ be a variation of polarized complex Hodge structure of weight~$0$ (say) on $U$. The associated local system $\ccV$, being semi-simple, decomposes as $\ccV=\bigoplus_{\alpha\in A}\ccV_\alpha^{n_\alpha}$, where $\ccV_\alpha$ are irreducible and pairwise non isomorphic. Similarly, $(V,\nabla)=\bigoplus_{\alpha\in A}(V_\alpha,\nabla)^{n_\alpha}$, and the polarization $k$, being $\nabla$-horizontal, decomposes with respect to $\alpha\in A$.

\begin{proposition}[{\cite[Prop\ptbl1.13]{Deligne87}}]\label{prop:Deligne}
For each $\alpha\in A$, there exists a unique (up to a shift of the filtration) variation of polarized complex Hodge structure $(V_\alpha,F^\cbbullet V_\alpha,\nabla,k_\alpha)$ of weight $0$ and a polarized complex Hodge structure $(H_\alpha,\bigoplus H_\alpha^\cbbullet,k^o_\alpha)$ of weight $0$ with $\dim H_\alpha=n_\alpha$ such that
$$
(V,F^\cbbullet V,\nabla,k)=\bigoplus_{\alpha\in A}\Big((H_\alpha,\bigoplus H_\alpha^\cbbullet,k^o_\alpha)\otimes(V_\alpha,F^\cbbullet V_\alpha,\nabla,k_\alpha)\Big).\eqno\qed
$$
\end{proposition}

\section{Hodge properties of the middle convolution}

\subsection{Behaviour of Hodge numerical data by middle convolution}

Let $\cV$ be a non-zero irreducible non-constant local system on $\Afu\setminus\bmx$ which underlies a variation of polarizable complex Hodge structure $(V,F^\cbbullet V,\nabla)$, and let $(M,F^\cbbullet M)$ be the associated complex Hodge module (a notion explained in \S\ref{subsec:complexHodgemodule} below).

\begin{proposition}\label{prop:MCchi1}
With these assumptions, $\MC_\chi(M)$ underlies a natural polarizable complex Hodge module and if $\chi,\chi'\neq1$ and $\chi=\exp(-\twopii \alpha_o)$ with $\alpha_o\in(0,1)$ (and similarly for $\chi'$),
\[
\MC_{\chi'}\MC_\chi(M,F^\cbbullet M)\simeq
\begin{cases}
\MC_{\chi'\chi}(M,F^\cbbullet M)(-1)&\text{if }\alpha_o+\alpha'_o\in(0,1],\\
\MC_{\chi'\chi}(M,F^\cbbullet M)&\text{if }\alpha_o+\alpha'_o\in(1,2).
\end{cases}
\]

\end{proposition}

Our objective is to prove the following.

\begin{theoreme}\label{th:MCchi}
Under Assumptions \ref{ass:irrednotO}\eqref{ass:irrednotO1} and \eqref{ass:irrednotO2} on $M$, and for $\chi=\lambda_o$,
\begin{enumerate}
\item\label{th:MCchi2}
$h^p(\MC_\chi(M))=h^p H^1(\DR M)$ (given by Formula \eqref{eq:nupH1*}),
\item\label{th:MCchi3}
Set $\lambda_o\!=\!\exp(-\twopii\alpha_o)$ with $\alpha_o\!\in\!(0,1)$. For $i\!=\!1,\dots,r$, $\lambda\!=\!\exp(-\twopii\alpha)\!\in\! S^1$ and $\ell\in\NN$, we have (together with a similar formula without $\ell$):
\[
\mu_{x_i,\lambda,\ell}^p(\MC_\chi(M))=\begin{cases}
\mu_{x_i,\lambda/\lambda_o,\ell}^{p-1}(M)&\text{if }\alpha\in(\alpha_o,1)\cup\{0\},\\[2pt]
\mu_{x_i,\lambda/\lambda_o,\ell}^p(M)&\text{if }\alpha\in(0,\alpha_o].
\end{cases}
\]
\item\label{th:MCchi4}
With the same assumptions, we have
\[
\delta^p(\MC_\chi(M))=\delta^p(M)+h^p(M)-\sum_{i=1}^r\bigg(\mu_{x_i,1}^p(M)+\sum_{\alpha\in(0,1-\alpha_o)}\mu_{x_i,\lambda}^{p-1}(M)\bigg).
\]
\end{enumerate}
\end{theoreme}

\begin{remarque}[On duality]\label{rem:duality}
Given a variation of polarized complex Hodge structure $(V,F^\cbbullet V,\nabla,k)$, its dual is of the same kind. If $(M,F^\cbbullet M)$ is the complex Hodge module (in the sense of \S\ref{subsec:complexHodgemodule} below) corresponding to $(V,F^\cbbullet V,\nabla,k)$, we denote by $\bD(M,F^\cbbullet)$ the complex Hodge module corresponding to the dual variation. The behaviour of duality by tensor product is clear. Noting that, as a unitary variation of Hodge structure, the dual of $L_\chi$ is $L_{\chi^{-1}}$, we conclude, by using the standard results of the behaviour of duality with respect to push-forward \cite{MSaito86}, that the dual of $\MC_{\chi}(M,F^\bbullet)$ (as defined by Proposition \ref{prop:MCchi1}) is isomorphic to $\MC_{\chi^{-1}}(\bD(M,F^\cbbullet M))$.

Assume that we are given an isomorphism $\omega:(V,F^\cbbullet V,\nabla,k)\isom(V,F^\cbbullet V,\nabla,k)^\vee$, hence $\omega:(M,F^\cbbullet M)\isom\bD(M,F^\cbbullet M)$. Then, setting $\chi=-1$, we obtain an isomorphism $\MC_{-1}(\omega):\MC_{-1}(M,F^\cbbullet M)\isom\bD\MC_{-1}(M,F^\cbbullet M)$.

If $\omega$ is $\pm$-symmetric, then $\MC_{-1}(\omega)$ is $\mp$-symmetric, due to the skew-symmetry of the Poincaré duality on $H^1$.
\end{remarque}

\subsection{Polarizable complex Hodge modules}\label{subsec:complexHodgemodule}
The theory of pure (or mixed) Hodge modules of M\ptbl Saito \cite{MSaito86,MSaito87}, originally written for objects with a $\QQ$-structure, extends naturally to the case of a $\RR$-structure.

\begin{definition}
A filtered $\cD_X$-module $(\ccM,F^\cbbullet\ccM)$ is a polarizable complex Hodge module on a complex manifold $X$ if it is a direct summand of a filtered $\cD_X$-module which underlies a polarizable real Hodge module which is pure of some weight.
\end{definition}

We can therefore apply various results of the theory of polarizable real Hodge modules to the complex case, almost by definition, since most operations (nearby cycles, vanishing cycles, grading by the weight filtration of the monodromy, push-forward by a projective morphism) are compatible with taking a direct summand.

Given a variation of polarized complex Hodge structure of weight $w$ as defined in \S\ref{subsec:VPCHS}, for which we now denote the grading by $\bigoplus_pH^{p,w-p}$, the $C^\infty$ complex conjugate bundle $\ov H$ is endowed with the grading $\bigoplus_q\ov H^{q,w-q}:=\bigoplus_q\ov{H^{w-q,q}}$ and with the flat connection $\ov D$ and adjoint pairing $k^*$, making it a variation of polarized complex Hodge structure of weight $w$. The direct sum $H\oplus\ov H$ underlies then a variation of polarized real Hodge structure of weight $w$. Therefore, a variation of polarizable complex Hodge structure is a smooth polarizable complex Hodge module.

The converse is also true if $X$ is quasi-projective, according to \cite[Prop\ptbl1.13]{Deligne87} (\cf Proposition \ref{prop:Deligne}). We will mainly use the case where $X=\Afu\setminus\bmx$.

The theorem of Schmid (Th\ptbl\ref{th:Schmid}), when applied to a real variation, produces a polarizable real Hodge module on $\Delta$.

\begin{corollaire}
The filtered $\cD_\Delta$-module \eqref{eq:goodfiltr} is a polarizable complex Hodge module.\qed
\end{corollaire}

\subsubsection*{Thom-Sebastiani}
We review here the main result of \cite{MSaito90-11} in our particular situation. We consider the local setup of \S\ref{subsec:localHodgetheory}. Let $(M_1,F^\cbbullet M_1),(M_2,F^\cbbullet M_2)$ be complex Hodge modules on the disc $\Delta$, where the filtrations are defined by \eqref{eq:goodfiltr}, having a singularity at the origin of the disc at most. We consider the product space $\Delta\times\Delta$ with coordinates $(t_1,t_2)$ and the sum map $s:\Delta\times\Delta\to\Delta$, $(t_1,t_2)\mto t=t_1+t_2$. The moderate vanishing cycle functor $\phi_s$ along the fibre $s=0$ is defined by using the $V$-filtration of Kashiwara-Malgrange (in dimension two). Since the singular locus of the external product $M_1\boxtimes M_2$ is contained in $\{t_1t_2=0\}$, the support of $\phi_s(M_1\boxtimes M_2)$ is reduced to $(0,0)$.

Note that our definition for the filtration on the nearby/vanishing cycles corresponds to that of \cite[(0.3), (0.4)]{MSaito86}.

Given $\lambda\in S^1$, we set $\lambda=\exp(-\twopii\beta)$ with $\beta\in(0,1]$. With this in mind, we have:

\begin{theoreme}[{\cite[Th\ptbl5.4]{MSaito90-11}}]\label{th:SaitoTS}
\[
\gr^p_F\phi_{s,\lambda}(M_1\boxtimes M_2)=\bigoplus_{\substack{(\lambda_1,\lambda_2)\\\lambda_1\lambda_2=\lambda}}\begin{cases}
\dpl\bigoplus_{j+k=p-1}\gr^j_F\phi_{t_1,\lambda_1}M_1\otimes\gr^k_F\phi_{t_2,\lambda_2}M_2\\[-10pt]
&\hspace*{-1cm}\text{if }\beta_1+\beta_2\in(0,1],\\[3pt]
\dpl\bigoplus_{j+k=p}\gr^j_F\phi_{t_1,\lambda_1}M_1\otimes\gr^k_F\phi_{t_2,\lambda_2}M_2\\[-10pt]
&\hspace*{-1cm}\text{if }\beta_1+\beta_2\in(1,2].
\end{cases}
\]
\end{theoreme}

Particular cases of this result have been obtained in \cite[Cor.\,6.2.3]{D-L99}. We will be mostly interested in the case where $M_2=L_\chi$ ($\chi=\lambda_o\neq1$) with its filtration making it of weight $0$, in which case the formula of \loccit makes precise the monodromy weight filtration in a simple way:\enlargethispage{\baselineskip}%
\begin{equation}\label{eq:SaitoTS}
\gr^p_F\rP_\ell\phi_{s,\lambda}(M\boxtimes L_\chi)=\begin{cases}
\gr^{p-1}_F\rP_\ell\phi_{t,\lambda/\lambda_o}M&\text{if }\beta\in(\alpha_o,1],\\[3pt]
\gr^p_F\rP_\ell\phi_{t,\lambda/\lambda_o}M
&\text{if }\beta\in(0,\alpha_o).
\end{cases}
\end{equation}

\begin{remarque}
An alternative proof of \eqref{eq:SaitoTS} by using the Fourier transformation would be possible, but would need the extension of Hodge theory to integrable twistor theory and would use the formulas given in \cite[Prop\ptbl4.5 and Cor\ptbl4.6]{Bibi08}. However, having at hand the results of \cite{MSaito90-11}, we did not make explicit such a proof.
\end{remarque}

\subsection{Proofs}
\begin{proof}[Proof of Proposition \ref{prop:MCchi1}]
We regard $L_\chi$ (trivial filtration $F^pL_\chi=L_\chi$ for $p\leq0$ and $0$ otherwise) as a polarizable complex Hodge module. We first notice that $(M,F^\cbbullet M)\boxtimes\nobreak L_\chi$ is a polarizable complex Hodge module on $\Afu_x\times\Afu_y$, according to \cite[Th\ptbl3.28]{MSaito87}. Using the notation of \S\ref{subsec:quickreview} and Proposition \ref{prop:MCchiconvolution}, $\ccM_\chi$ underlies a polarizable complex Hodge module with filtration $F^\cbbullet\ccM_\chi$ defined by a formula similar to \eqref{eq:goodfiltr} along $\infty\times\Afu_t$. We now define $(\MC_\chi(M),F^\cbbullet\MC_\chi(M))$ as the pushforward $\cH^0\wt s_+(\ccM_\chi,F^\cbbullet\ccM_\chi)$. By \cite[Th\ptbl5.3.1]{MSaito86}, it is a polarizable complex Hodge module, and so, as remarked above, it corresponds to a variation of polarizable complex Hodge structure on $\Afu\setminus\bmx$.

Before proving the second statement, let us compute $L_\chi\star L_{\chi'}$ ($\chi,\chi'\neq1$, $\star=\nobreak\star_!,\star_*,\star_\Mid$) as a polarizable complex Hodge module. Assume first that $\chi\chi'\neq1$. Then the underlying holonomic module is $L_{\chi\chi'}$, according to \eqref{eq:LLchi}. On the other hand, for $t_o\neq0$, $L_\chi\otimes L_{\chi',t_o}$ has a non-trivial monodromy at infinity, where $L_{\chi',t_o}$ denotes the Kummer module $L_{\chi'}$ translated so that its singularities are at~$t_o$ and~$\infty$, so its various extensions ($!,*,\Mid$) all coincide, and so do the various $L_\chi\star L_{\chi'}$ as complex Hodge modules. The fibre of $L_\chi\star L_{\chi'}$ at $t_o$ is $H^1\big(\PP^1,\DR(L_\chi\otimes L_{\chi',t_o})_{\min}\big)$ together with its Hodge filtration. Now $L_\chi\otimes L_{\chi',t_o}$ is a unitary rank-one local system with singular points $0,t_o,\infty$ and respective monodromies $\chi,\chi',(\chi\chi')^{-1}$. Formula~\eqref{eq:nupH1} gives then (with the notation as in the statement):
\[
\big(L_\chi\star L_{\chi'},F^\cbbullet(L_\chi\star L_{\chi'})\big)=
\begin{cases}
(L_{\chi\chi'},F^\cbbullet L_{\chi\chi'})(-1)&\text{if }\alpha_o+\alpha'_o\in(0,1),\\
(L_{\chi\chi'},F^\cbbullet L_{\chi\chi'})&\text{if }\alpha_o+\alpha'_o\in(1,2).
\end{cases}
\]
Assume now that $\chi\chi'=1$. Then $L_\chi\star L_{\chi^{-1}}=L_1=\delta_0$ is supported at the origin, and we wish to compare the filtrations. The comparison, together with the twist by $-1$, directly follows from the first formula in Theorem \ref{th:SaitoTS}, with $\beta_1+\beta_2=1$.

The second statement, which holds in a more general setting, is proved by considering the category of complex mixed Hodge modules. As above, we can reduce to considering real mixed Hodge modules. Then the framework of \cite{MSaito87} allows us to apply the arguments of Proposition \ref{prop:starmidcomm} by considering this abelian category instead of that of holonomic modules, since the functors $\star_*L_\chi,\star_!L_\chi,\star_\Mid L_\chi$ are endo-functors of this category. In all cases we obtain the associativity property in this category: $(M\star_\Mid L_\chi)\star_\Mid L_{\chi'}=M\star_\Mid(L_\chi\star_\Mid L_{\chi'})$. For the statement when $\chi\chi'=1$, we use that if a morphism in the category of mixed Hodge modules is epi (\resp mono) in the category of holonomic module, it is also epi (\resp mono) in the category of mixed Hodge modules.

Our previous computation of $L_\chi\star_\Mid L_{\chi'}$ concludes the proof.
\end{proof}

\begin{proof}[Proof of Theorem \ref{th:MCchi}\eqref{th:MCchi2}]
For $t_o\in\Afu\setminus\bmx$, the fibre of $(\MC_\chi(M),F^\cbbullet\MC_\chi(M))$ at $t_o$ is equal to $H^1\big(\PP^1,\DR((M\otimes L_{\chi,t_o})_{\min})\big)$ together with its Hodge filtration, where $L_{\chi,t_o}$ denotes the Kummer module $L_\chi$ translated so that its singularities are at~$t_o$ and~$x_{r+1}$. We notice that the variation of polarize complex Hodge structure that $\cV\otimes \cL_{\chi,t_o}$ underlies has the same characteristic numbers as the one that $\cV$ underlies, except that it has singularities $\bmx\cup\{t_o\}$ instead of $\bmx\cup\{x_{r+1}\}$: the assertion is clear for the local characteristic numbers, and for the $\delta^p$'s one uses Proposition \ref{prop:degreetensor}. Then, according to \eqref{eq:nupH1}, we have
\begin{align*}
h^p(\MC_\chi(M))=h^pH^1\big(\PP^1,\DR((M&\otimes L_{\chi,t_o})_{\min})\big)\\
&=h^pH^1(\PP^1,\DR\ccM)=h^pH^1(\DR M).\qedhere
\end{align*}
\end{proof}

\begin{proof}[Proof of Theorem \ref{th:MCchi}\eqref{th:MCchi3}]
Recall that the nearby cycle functor is compatible with projective push-forward, as follows from \cite[Prop\ptbl3.3.17]{MSaito86}. Moreover, $\phi_{\wt s-x_i}\ccM_\chi$ is supported at $(x_i,x_i)\in\Afu_x\times\Afu_t$. Therefore, $\phi_{x_i}(\MC_\chi(M),F^\cbbullet\MC_\chi(M))$ can be computed as $\phi_{\wt s-x_i}(\ccM_\chi,F^\cbbullet\ccM_\chi)$ (see Footnote \ref{footnote:1} for the notation). We will use the analytic topology locally at $(x_i,x_i)$ to do this computation, and we will assume that $x_i=0$ to simplify the notation. Then the result is given by \eqref{eq:SaitoTS}.
\end{proof}

\begin{proof}[Proof of Theorem \ref{th:MCchi}\eqref{th:MCchi4}]
Set $\gamma^p=\delta^p-\delta^{p-1}$. By \ref{th:MCchi}\eqref{th:MCchi2}--\eqref{th:MCchi3} and \eqref{eq:nupH1*}, we get
\begin{align*}
h^p(\MC_\chi(M))+h^p(M)&=-\gamma^p(M)+\sum_{i=1}^r\big[\mu_{x_i,\neq1}^{p-1}(M)+\mu_{x_i,1}^p(M)\big]\\[-3pt]
h^p(\MC_\chi(M))+h^{p-1}(M)&=-\gamma^p(\MC_\chi(M))+\sum_{i=1}^r\big[\mu_{x_i,\neq1}^{p-1}(\MC_\chi(M))\\[-8pt]
&\hspace*{5cm}+\mu_{x_i,1}^p(\MC_\chi(M))\big]\\
&\hspace*{-3cm}=-\gamma^p(\MC_\chi(M))+\sum_{i=1}^r\bigg(\sum_{\alpha\in(0,1-\alpha_o]}\mu_{x_i,\lambda}^{p-2}(M)+\sum_{\alpha\in(1-\alpha_o,1)\cup\{0\}}\mu_{x_i,\lambda}^{p-1}(M)\bigg).
\end{align*}
Hence
\begin{multline*}
\gamma^p(\MC_\chi(M))=\gamma^p(M)+h^p(M)-h^{p-1}(M)\\[-3pt]
-\sum_{i=1}^r\bigg([\mu_{x_i,1}^p(M)-\mu_{x_i,1}^{p-1}(M)\big]+\sum_{\alpha\in(0,1-\alpha_o)}[\mu_{x_i,\lambda}^{p-1}(M)-\mu_{x_i,\lambda}^{p-2}(M)]\bigg).
\end{multline*}
Summing these equalities for $p'\leq p$ gives the desired formula.
\end{proof}

\section{Examples}
In the following examples (except in Lemma \ref{lemme4.0.1}), we will consider local systems on the punctured Riemann sphere $\PP^1\setminus \{x_1,\dots,x_4\}$. We will set $x_4=\infty$, and we will assume that $x_1,x_2,x_3$ are distinct points at finite distance, so that we have here $r=3$. A rank-one local system $\cL$ on $\PP^1\setminus \{x_1,\dots,x_4\}$ will simply be denoted by $(\lambda_1,\lambda_2,\lambda_3,\lambda_4)$, where the $\lambda_j$'s are the local monodromies. For a local system of higher rank, we use the following notation for the local monodromy matrices: a unipotent Jordan block of length $m$ is denoted $\rJ(m)$, then $\lambda\rJ(m)$ denotes the length\nobreakdash-$m$ Jordan block with eigenvalue $\lambda$ (we will use the notation $-\rJ(m)$ if $\lambda=\nobreak-1$), and $\lambda_1\rJ(m_1)\oplus \lambda_2\rJ(m_2)$ denotes a matrix in Jordan canonical form consisting the corresponding two Jordan blocks, etc.  Lastly, the identity matrix in $\GL_m(\CC)$ is denoted~$1_m$.

\subsection{A physically rigid $G_2$-local system with minimal Hodge
filtration}\label{subsec:exam1}
Set $\varphi=\exp(\sfrac{\twopii}{3})\in \CC$ be the third root of unity in the upper half plane. It follows from \cite{D-R10}, Theorem 1.3.2, that there exists a physically rigid $G_2$-local system $\cG$ on $\PP^1\setminus \{x_1,x_2,x_3,x_4\}$ with the following local monodromy at $x_1,\dots,x_4$ (respectively):
\begin{equation}\label{eq:locmonoG}
-1_4\oplus 1_3,\quad \rJ(2)\oplus \rJ(3)\oplus \rJ(2),\quad \varphi\rJ(3)\oplus
1\oplus \bar{\varphi}\rJ(3),\quad {\rm 1}_7.
\end{equation}
This local system can be constructed using the following sequence of middle convolutions and tensor products: Set
\begin{align*}
\cL_0&:=(-1,-\bar{\varphi},-\bar{\varphi},-\bar{\varphi}),&
\cL_1&:=(1,-1,-1,1),& \cL_2&:=\cL(-1,1,-\bar{\varphi},\varphi),\\
\cL_3&:=(1,-\bar{\varphi},-\varphi,1),&
\cL_4&:=(-1,1,-\varphi,\bar{\varphi}),&
\cL_5&:=(1,-\bar{\varphi},-\varphi,1),\\
\cL_6&:=(-1,1,\bar{\varphi},-\varphi).
\end{align*}
Then $\cG$ is isomorphic to
$$
\cL_6\otimes \MC_{-\varphi}(\cL_5\otimes\MC_{-\bar{\varphi}}(\cL_4\otimes \MC_{-1}(\cL_3\otimes\MC_{-1}(\cL_2\otimes\MC_{-\varphi}(\cL_1\otimes \MC_{-\bar{\varphi}}(\cL_0)))))).
$$
(That the local monodromy of $\cG$ defined by the formula above is given by \eqref{eq:locmonoG} will be a byproduct of the proof of Theorem \ref{thm:ex2}.) If $L_i$ denotes the unitary rank-one-variation of complex Hodge structure (trivial filtration) underlying $\cL_i$ ($i=0,\dots,6$), then one constructs a variation of polarized complex Hodge structure~$G$ underlying $\cG$ as
$$
L_6\otimes \MC_{-\varphi}(L_5\otimes\MC_{-\bar{\varphi}}(L_4\otimes \MC_{-1}(L_3\otimes\MC_{-1}(L_2\otimes\MC_{-\varphi}(L_1\otimes \MC_{-\bar{\varphi}}(L_0))))))).
$$
We note that his variation is real: indeed, the local monodromies being defined over~$\RR$ by the formula above, the local real structures $\cG\isom\ov\cG$ come from a unique global isomorphism, which is a real structure on $\cG$. By Proposition \ref{prop:Deligne}, it is compatible with the Hodge filtration up to a shift, hence giving rise to a variation of real Hodge structure.

\begin{theoreme}\label{thm:ex2}
The Hodge data of $G$ are as follows:
$$\arraycolsep4.5pt\renewcommand{\arraystretch}{1.4}
\begin{array}{|c||c||c|c|c|c|c||c|}
\hline
p&h^p& \mu_{x_1,-1,0}^p&\mu_{x_2,1,1}^p&\mu_{x_2,1,0}^p&
\mu_{x_3,\varphi,2}^p
&\mu_{x_3, \bar{\varphi},2}^p&\delta^p\\
\hline
2&2&1&0&0&0&0&-2\\
\hline
3&3&2&0&1&0&0&-2\\
\hline
4&2&1&1&1&1&1&-1\\
\hline
\end{array}
$$\end{theoreme}

\begin{proof}
We will use Formulas \eqref{eq:nupVL}--\eqref{eq:muVL}, Proposition \ref{prop:degreetensor}, Formula \eqref{eq:nupH1*} and Theorem \ref{th:MCchi}. Let us make explicit the first steps.

The Hodge data of $\cL_0$ are as follows:
$$
h^0=1;\quad \mu_{x_1,-1,0}^0=\mu_{x_2,-\bar{\varphi},0}^0=\mu_{x_3,-\bar{\varphi},0}^0=1;\quad
\delta^0=-3
$$
The degree is computed with the residue formula: The residue $a_i$ of the canonical extension of the connection $\nabla_0$ underlying $\cL_0=(-1,-\bar{\varphi},-\bar{\varphi},-\bar{\varphi})$ at $x_1,\dots,x_4$, respectively is given by (recall the convention that $\alpha_i=\exp(-\twopii a_i)$):
$$
a_1=\frac{1}{2},\quad a_2=\frac{5}{6},\quad a_3=\frac{5}{6},\quad
a_4=\frac{5}{6},
$$
so $\delta^0=-\sum_ia_i=-3$. It follows from Proposition \ref{prop:numdataMCchi} that the middle convolution $\MC_{-\bar{\varphi}}(L_0)$ has rank $2$ and the following local monodromy at $x_1,x_2,x_3, x_4$ (respectively):
$$
\bar{\varphi}\oplus 1,\quad \varphi\oplus 1 ,\quad\varphi\oplus 1, \quad-\varphi\cdot 1_2.
$$
By Theorem \ref{th:MCchi}(1) and Equation~\eqref{eq:nupH1*}, the Hodge number $h^0$ of $\MC_{-\bar{\varphi}}(L_0)$ is
$$
h^0(\MC_{-\bar{\varphi}}(L_0))= -\delta^0(L_0)-h^0(L_0)=-(-3)-1=2.
$$
This implies that all other Hodge numbers $h^p(\MC_{-\bar{\varphi}}(L_0))$ vanish. Hence, we obtain the following local Hodge data $\mu_{x_i,\alpha,\ell}^0$ for $\MC_{-\bar{\varphi}}(L_0)$
$$
\mu_{x_1,\bar{\varphi},0}^0=\mu_{x_2,\varphi,0}^0=\mu_{x_3,\varphi,0}^0=1
$$
and all other $\mu_{x_i,\alpha,\ell}^p$'s vanish. By Theorem \ref{th:MCchi}(3), the degree $\delta^0$ of $\MC_{-\bar{\varphi}}(L_0)$ is
$$
\delta^0=\delta^0(L_0)+h^0(L)-\sum_{i=1}^3\bigg(\mu_{x_i,1}^0+
\sum_{\alpha\in (0,\sfrac{1}{6})}\mu_{x_i,\exp(-\twopii \alpha)}^{-1}(L_0)\bigg)=-3+1=-2.
$$

Let us now consider $L_1\otimes\MC_{-\bar{\varphi}}(L_0)$. Its local monodromy is
$$
\bar{\varphi}\oplus 1,\quad -\varphi\oplus -1 ,\quad-\varphi\oplus -1, \quad-\varphi\cdot 1_2.
$$
This implies that the nonvanishing local Hodge data of $L_1\otimes
\MC_{-\bar{\varphi}}(L_0)$ are
$$
h^0=2;\quad \mu_{x_1,\bar{\varphi},0}^0=\mu_{x_2,-\varphi,0}^0=\mu_{x_2,-1,0}^0=\mu_{x_3,\varphi,0}^0=\mu_{x_3,-1,0}^0=1.
$$
By Proposition \ref{prop:degreetensor}, the only nonvanishing Hodge degree of $L_1\otimes\MC_{-\bar{\varphi}}(L_0)$ is
\begin{align*}
\delta^0(L_1\otimes
\MC_{-\bar{\varphi}}(L_0))&=\delta^0(\MC_{-\bar{\varphi}}(L_0))-h^0(\MC_{-\bar{\varphi}}(L_0))\\
&\hspace*{3cm}+\sum_{i=2,3}\sum_{\beta\in [-\sfrac{1}{2},0)}
\mu_{x_i,\exp(-\twopii\beta)}^0\\
&=-2-2+(1+1)=-2.
\end{align*}
It follows from Proposition \ref{prop:numdataMCchi} that $\MC_{-\varphi}(L_1\otimes \MC_{-\bar{\varphi}}(L_0))$ has rank $3$ and the following local monodromy at $x_1,x_2,x_3, x_4$ (respectively):
$$
-1\oplus 1_2,\quad \bar{\varphi}\oplus \varphi\oplus 1,\quad \bar{\varphi}\oplus \varphi\oplus 1,\quad -\bar{\varphi}\cdot 1_3.
$$
Using Theorem \ref{th:MCchi}(1), we have
\begin{align*}
h^1(\MC_{-\varphi}(L_1\otimes \MC_{-\bar{\varphi}}(L_0)))&=
\delta^0-\delta^1-h^1+\sum_{i=1}^3(\mu_{x_i,\neq 1}^0+\mu_{x_i,1}^1)\\
&=-2-0-0+(1+2+2)=3,
\end{align*}
where we use the convention that the Hodge data on the right hand side are those of the convolutant $H_1:=L_1\otimes \MC_{-\bar{\varphi}}(L_0)$. Consequently, the other local Hodge data of $\MC_{-\varphi}(H_1)$ are
$$
\mu_{x_1,-1,0}^1=\mu_{x_2,\bar{\varphi},0}^1=\mu_{x_2,\varphi,0}^1 =\mu_{x_3,\bar{\varphi},0}^1=\mu_{x_3,\varphi,0}^1 =1.
$$
By Theorem \ref{th:MCchi}(3),
\begin{align*}
\delta^1(\MC_{-\varphi}(H_1))&=
\delta^1+h^1-\sum_{i=1}^3\bigg(\mu_{x_i,1}^1+\sum_{\alpha\in (0,1-\alpha_0),\,\alpha_0=\sfrac{1}{6}}\mu_{x_i,\exp(-\twopii\alpha)}^0\bigg)\\
&= 0+0-(1+2+2)=-5
\end{align*}
(again, we use the convention that the Hodge data on the right hand side are those of $H_1$).
The local monodromy of $H_2:=L_2\otimes \MC_{-\varphi}(H_1)$ is
$$
1\oplus -1_2,\quad \bar{\varphi}\oplus \varphi\oplus 1,\quad -{\varphi}\oplus -1\oplus -\bar{\varphi},\quad - 1_3.
$$
Consequently, the local Hodge data of $H_2$ are
$$
h^1=3;\quad \mu_{x_1,-1,0}^1=2,\quad \mu_{x_2,\bar{\varphi},0}^1=\mu_{x_2,\varphi,0}^1=
\mu_{x_3,-{\varphi},0}^1=\mu_{x_3,-1,0}^1=\mu_{x_3,-\bar{\varphi},0}^1=1.
$$
By Proposition \ref{prop:degreetensor}, the only nonvanishing global Hodge number of $H_2$ is (with $\alpha_1=\frac{1}{2}, \, \alpha_2=0,\,\alpha_3=\frac{5}{6},\, \alpha_4=\frac{4}{6}$):
\begin{align*}
\delta^1(H_2)&=-5+3\cdot(-2)+\sum_{i=1,3,4}\sum_{\alpha\in [1-\alpha_i,1)}\mu_{x_i,\exp(-\twopii\alpha)}^p \\
&=-11+(1+2+3)=-5.
\end{align*}
In the following, we only list the Hodge data for the sheaves involved in the construction of $H$. Each value can be immediately verified from the preceding Hodge data, using Proposition \ref{prop:numdataMCchi} for the local monodromy, Proposition \ref{prop:degreetensor} and Equations~\eqref{eq:nupVL}--\eqref{eq:muVL} for the tensor product, and Theorem \ref{th:MCchi} for the middle convolution.

$\MC_{-1}(H_2)$:
\begin{gather*}
\rJ(2)\oplus \rJ(2),\quad -\bar{\varphi}\oplus -\varphi\oplus 1_2,\quad {\varphi}\oplus \rJ(2) \oplus \bar{\varphi},
\quad -1_4,\\
\arraycolsep4.5pt\renewcommand{\arraystretch}{1.4}
\begin{array}{|c||c||c|c|c|c|c|c||c|}
\hline
p&h^p& \mu_{x_1,1,0}^p&\mu_{x_2,-\bar{\varphi},0}^p&\mu_{x_2,-{\varphi},0}^p&\mu_{x_3,\varphi,0}^p&\mu_{x_3,1,0}^p
&\mu_{x_3, \bar{\varphi},0}^p&\delta^p\\
\hline
1&2&0&0&1&0&0&1&-2\\
\hline
2&2&2&1&0&1&1&0&-2\\
\hline
\end{array}
\end{gather*}
\hspace{.5cm}

$H_3:=L_3\otimes \MC_{-1}(H_2)$:
\begin{gather*}
\rJ(2)\oplus \rJ(2),\quad
{\varphi}\oplus 1 \oplus- \bar{\varphi}1_2,
\quad
-\bar{\varphi}\oplus -\varphi\rJ(2) \oplus -1,
\quad -1_4,\\
\arraycolsep4.5pt\renewcommand{\arraystretch}{1.4}
\begin{array}{|c||c||c|c|c|c|c|c||c|}
\hline
p&h^p& \mu_{x_1,1,0}^p&\mu_{x_2,{\varphi},0}^p&\mu_{x_2,-\bar{\varphi},0}^p&\mu_{x_3,-\bar{\varphi},0}^p&\mu_{x_3,-\varphi,1}^p
&\mu_{x_3, -1,0}^p&\delta^p\\
\hline
1&2&0&0&1&0&0&1&-3\\
\hline
2&2&2&1&1&1&1&0&-3\\
\hline
\end{array}
\end{gather*}
\hspace{.5cm}

$\MC_{-1}(H_3)$:
\begin{gather*}
-1_2\oplus 1_3, \quad
1_2\oplus  -{\varphi}\oplus \bar{\varphi}1_2,
\quad
\bar{\varphi}\oplus \varphi\rJ(2) \oplus \rJ(2),
\quad -1_5,\\
\arraycolsep4.5pt\renewcommand{\arraystretch}{1.4}
\begin{array}{|c||c||c|c|c|c|c|c||c|}
\hline
p&h^p& \mu_{x_1,-1,0}^p&\mu_{x_2,-{\varphi},0}^p&\mu_{x_2,\bar{\varphi},0}^p&\mu_{x_3,\bar{\varphi},0}^p&\mu_{x_3,\varphi,1}^p
&\mu_{x_3, 1,0}^p&\delta^p\\
\hline
1&1&0&0&1&0&0&0&-1\\
\hline
2&3&2&1&1&1&0&1&-4\\
\hline
3&1&0&0&0&0&1&0&-1\\
\hline
\end{array}
\end{gather*}
\hspace{.5cm}

$H_4:=L_4\otimes \MC_{-1}(H_3)$:
\begin{gather*}
-1_2\oplus 1_3, \quad
1_2\oplus  -{\varphi}\oplus \bar{\varphi}1_2,
\quad
-1\oplus -\bar{\varphi}\rJ(2) \oplus -\varphi \rJ(2),
\quad -\bar{\varphi}\cdot 1_5,\\
\arraycolsep4.5pt\renewcommand{\arraystretch}{1.4}
\begin{array}{|c||c||c|c|c|c|c|c||c|}
\hline
p&h^p& \mu_{x_1,-1,0}^p&\mu_{x_2,-{\varphi},0}^p&\mu_{x_2,\bar{\varphi},0}^p&\mu_{x_3,-1,0}^p&\mu_{x_3,-\bar{\varphi},1}^p
&\mu_{x_3, -\varphi,1}^p&\delta^p\\
\hline
1&1&1&0&1&0&0&0&-2\\
\hline
2&3&1&1&1&1&0&1&-5\\
\hline
3&1&1&0&0&0&1&0&-2\\
\hline
\end{array}
\end{gather*}
\hspace{.5cm}

$\MC_{-\bar{\varphi}}(H_4)$:
\begin{gather*}
1_3\oplus \bar{\varphi}1_3, \quad
\rJ(2) \oplus -{\varphi}1_2\oplus 1_2,
\quad
\bar{\varphi}\oplus {\varphi}\rJ(2) \oplus \rJ(3),
\quad - {\varphi}\cdot 1_6,\\
\arraycolsep4.5pt\renewcommand{\arraystretch}{1.4}
\begin{array}{|c||c||c|c|c|c|c|c||c|}
\hline
p&h^p& \mu_{x_1,\bar{\varphi},0}^p&\mu_{x_2,-1,0}^p&\mu_{x_2,-{\varphi},0}^p&\mu_{x_3,\bar{\varphi},0}^p&\mu_{x_3,{\varphi},1}^p
&\mu_{x_3, 1,1}^p&\delta^p\\
\hline
1&1&1&0&1&0&0&0&-1\\
\hline
2&3&1&0&1&1&0&0&-2\\
\hline
3&2&1&1&0&0&1&1&-1\\
\hline
\end{array}
\end{gather*}
\hspace{.5cm}

$H_5:=L_5\otimes \MC_{-\bar{\varphi}}(H_4)$:
\begin{gather*}
1_3\oplus \bar{\varphi}1_3, \quad
-\bar{\varphi}\rJ(2)\oplus 1_2\oplus -\bar{\varphi}1_2,
\quad
-1\oplus -\bar{\varphi}\rJ(2) \oplus -\varphi\rJ(3),
\quad - {\varphi}\cdot 1_6,\\
\arraycolsep4.5pt\renewcommand{\arraystretch}{1.4}
\begin{array}{|c||c||c|c|c|c|c|c||c|}
\hline
p&h^p& \mu_{x_1,\bar{\varphi},0}^p&\mu_{x_2,-\bar{\varphi},1}^p&\mu_{x_2,-\bar{\varphi},0}^p&\mu_{x_3,-1,0}^p&
\mu_{x_3,-\bar{\varphi},1}^p
&\mu_{x_3, -\varphi,2}^p&\delta^p\\
\hline
1&1&1&0&0&0&0&0&-1\\
\hline
2&3&1&0&1&1&0&0&-4\\
\hline
3&2&1&1&1&0&1&1&-3\\
\hline
\end{array}
\end{gather*}
\hspace{.5cm}

$\MC_{-{\varphi}}(H_5)$:
\begin{gather*}
1_4\oplus -1_3, \quad
\rJ(3)\oplus \rJ(2)\oplus \rJ(2),
\quad
\varphi\oplus \rJ(3) \oplus \bar{\varphi}\rJ(3),
\quad - \bar{\varphi}\cdot 1_7,\\
\arraycolsep4.5pt\renewcommand{\arraystretch}{1.4}
\begin{array}{|c||c||c|c|c|c|c|c||c|}
\hline
p&h^p& \mu_{x_1,-1,0}^p&\mu_{x_2,1,1}^p&\mu_{x_2,1,0}^p&\mu_{x_3,\varphi,0}^p&
\mu_{x_3,1,1}^p
&\mu_{x_3, \bar{\varphi},2}^p&\delta^p\\
\hline
2&2&1&0&0&0&0&0&-3\\
\hline
3&3&1&0&1&1&0&0&-4\\
\hline
4&2&1&1&1&0&1&1&-2\\
\hline
\end{array}
\end{gather*}
\hspace{.5cm}

$G:=L_6\otimes \MC_{-{\varphi}}(H_5)$:
\begin{gather*}
-1_4\oplus1_3, \quad
\rJ(3)\oplus \rJ(2)\oplus \rJ(2),
\quad
\bar{\varphi}\rJ(3) \oplus1\oplus {\varphi}\rJ(3),
\quad 1_7,\\
\arraycolsep4.5pt\renewcommand{\arraystretch}{1.4}
\begin{array}{|c||c||c|c|c|c|c||c|}
\hline
p&h^p& \mu_{x_1,-1,0}^p&\mu_{x_2,1,1}^p&\mu_{x_2,1,0}^p&
\mu_{x_3,\varphi,2}^p
&\mu_{x_3, \bar{\varphi},2}^p&\delta^p\\
\hline
2&2&1&0&0&0&0&-2\\
\hline
3&3&2&0&1&0&0&-2\\
\hline
4&2&1&1&1&1&1&-1\\
\hline
\end{array}
\end{gather*}
\end{proof}

We remark that the Hodge filtration of $G$ has minimal length among the irreducible rank-$7$ local systems with $G_2$-monodromy which underlie a variation of polarized \emph{real} Hodge structure, according to \cite[Ch.\,IV]{G-G-K12} (\cf also \cite[\S9]{K-P12}).

\subsection{An orthogonally rigid $G_2$-local system with maximal Hodge filtration}\label{subsec:exam2}

By the classification of orthogonally rigid local systems with $G_2$-monodromy given in \cite{D-R12} there exists, up to tensor products with rank-one local systems, a unique $\ZZ$-local system $\cH$ on a $4$-punctured sphere $\PP^1\setminus \{x_1,\dots,x_4\}$ which is orthogonally rigid of rank $7$ and which satisfies the following properties (this is the case $P3, 6$ in \loccit):
\begin{itemize}
\item
The monodromy of $\cH$ is Zariski dense in the exceptional algebraic group~$G_2$.
\item
The underlying motive is defined over $\QQ$ (\ie we use only convolutions with quadratic Kummer sheaves and the tensor operations use motivic sheaves defined over $\QQ$).
\item
The local system $\cH$ does not have indecomposable local monodromy at any singular point.
\end{itemize}

The construction of $\cH$ is as follows. Start with a rank-one local system $\cL_0$ on $\Afu\setminus \{x_1,x_2,x_3\}$ with monodromy tuple $(-1,-1,-1,-1)$. Similarly, we define local systems of rank one $\cL_1,\cL_2,\cL_3$ on $\Afu\setminus \{x_1,x_2,x_3\}$ given by the monodromy tuples $(1,1,-1,-1)$, $(-1,1,1,-1)$, $(-1,1,-1,1)$, respectively. Let $-1:\pi_1(\GG_m(\CC))\to \CC^\times$ be the unique quadratic character. Then define $\cH$ as
$$
\cH= \MC_{-1}(\cL_3\otimes \MC_{-1}(\cL_2\otimes \tilde{\Lambda}^2(\MC_{-1}(\cL_1\otimes \Sym^2(\MC_{-1}(\cL_0)))) )),
$$
where $\tilde{\Lambda}^2$ is defined as follows. Notice first that $\cL_j$ ($j=0,\dots,3$) are self-dual unitary local systems. According to Remark \ref{rem:duality} and to Corollary \ref{cor:localdatareg} for the computation of the rank, $\MC_{-1}(\cL_0)$ is a symplectic irreducible variation of Hodge structure of rank two, and then $\Sym^2(\MC_{-1}(\cL_0))$ and $\cL_1\otimes \Sym^2(\MC_{-1}(\cL_0))$ are orthogonal irreducible variations of Hodge structure of rank three (see the proof of Theorem \ref{thmhodgeregular} for the argument concerning irreducibility). Then $\MC_{-1}(\cL_1\otimes\nobreak \Sym^2(\MC_{-1}(\cL_0)))$ is a symplectic irreducible variation of Hodge structure of rank four.

Now, any symplectic irreducible variation of Hodge structure $\cV$ of rank four having $\Sp_4$-monodromy has the property that the exterior square $\Lambda^2(\cV)$ is the direct sum of a rank-one variation $\cE$ (given by the symplectic form) and a variation $\tilde{\Lambda}^2(\cV)$ of rank five (this construction establishes the exceptional isomorphism $\Sp_4\simeq\nobreak \SO_5$). We will explain below that the local monodromy of $\tilde{\Lambda}^2(\cV)$ at $x_3$ is maximally unipotent. Since we a priori know that $\tilde{\Lambda}^2(\cV)$ is semi-simple, being a direct summand of $\cV\otimes\cV$, it is therefore irreducible.

It is proved in \cite{D-R12} that the local monodromy of $\cH$ at $x_1,\dots,x_4$ is given by
$$
\rJ(2)\oplus \rJ(3)\oplus \rJ(2),\quad 1_3\oplus
-\rJ(2)\oplus -\rJ(2),
\quad \rJ(2)\oplus -\rJ(3)\oplus \rJ(2),\quad -1_7,$$
respectively.

If $L_0,L_1,L_2,L_3$ denote the unitary variation of complex Hodge structure (with trivial Hodge filtration $F^0L_j=L_j$, $F^1L_j=0$) underlying $\cL_0,\cL_1,\cL_2,\cL_3$, then one obtains a rank-seven variation of polarized complex Hodge structure
$$
H= \MC_{-1}(L_3\otimes \MC_{-1}(L_2\otimes \tilde{\Lambda}^2(\MC_{-1}(L_1\otimes \Sym^2(\MC_{-1}(L_0)))) ))$$
whose underlying local system is $\cH$. It will be clear from the computation that each local system which $\MC_{-1}$ is applied to has scalar monodromy $-\id$ at $x_4$.

\begin{theoreme}\label{thmhodgeregular}
The variation of polarized Hodge structure $H$ has the following local and global Hodge data:
$$\arraycolsep4.5pt\renewcommand{\arraystretch}{1.4}
\begin{array}{|l|c|c|c|c|c|c|c|}
\hline
p&1&2&3&4&5&6&7\\
\hline\hline
h^p &1&1&1&1&1&1&1\\
\hline
\mu_{x_1,1,0}^p&0&1&0&0&0&0&1\\
\mu_{x_1,1,1}^p&0&0&0&0&1&0&0\\
\hline
\mu_{x_2,-1,1}^p&0&0&1&0&0&1&0\\
\hline
\mu_{x_3,-1,3}^p&0&0&0&0&1&0&0\\
\mu_{x_3,1,0}^p&0&1&0&0&0&0&1\\
\hline
\delta^p& -1&-1&-2&-1&-2&-1&0\\
\hline
\end{array}
$$
\end{theoreme}

We will need the following lemma:

\begin{lemme}\label{lemme4.0.1}
Let $(E,\nabla,F^\cbbullet E)$ be the filtered flat bundle underlying a polarizable variation of complex Hodge structure on $\PP^1\setminus\{0,1,\infty\}$. Set $x_1=0$, $x_2=1$, $x_3=\infty$. Assume the following:
\begin{enumerate}
\item
$(E,\nabla)$ is irreducible,
\item
$\rk E=3$,
\item
the monodromy at each singular point is maximally unipotent.
\end{enumerate}
Then the canonically extended Hodge bundles $H^p$ have degree $\delta^p(E)$
as follows:
\[
\delta^2(E)=1,\quad\delta^1(E)=0,\quad\delta^0(E)=-1.
\]
\end{lemme}

\begin{proof}
The local data must be as follows  (up to a shift of the filtration) because of the assumption on the local monodromies:
\begin{enumerate}
\item
$h^p(E)=1$ for $p=0,1,2$ and $h^p(E)=0$ otherwise,
\item
for $i=1,2$, $\mu_{x_i,1,1}^2=1$ and all other $\mu_{x_i,\lambda,\ell}^p$ are zero, hence $\mu_{x_i,1}^2=\mu_{x_i,1}^1=1$, all other $\mu_{x_i,\lambda}^p$ are zero,
\item
$\nu_{x_3,1,2}^2=1$ and all other $\nu_{x_3,\lambda,\ell}^p$ are zero, hence $\nu_{x_3,1,\prim}^2=1$, all other $\nu_{x_3,1,\prim}^p$ are zero.
\end{enumerate}
We have $\dim H^1(\PP^1,j_*\cE)\!=\!\sum_{i=1}^3\mu_{x_i}-2\rk E=0$. Therefore, according to \eqref{eq:nupH1},
\begin{align*}
0=h^3H^1(\PP^1,j_*\cE)&=\delta^2(E)-\nu_{x_3,1,\prim}^2(E)=\delta^2(E)-1\\
0=h^2H^1(\PP^1,j_*\cE)&=\delta^1(E)-\delta^2(E)-h^2(E)+\mu_{x_1,1}^2+\mu_{x_2,1}^2=\delta^1(E)-\delta^2(E)+1\\
0=h^1H^1(\PP^1,j_*\cE)&=\delta^0(E)-\delta^1(E)-h^1(E)+\mu_{x_1,1}^1+\mu_{x_2,1}^1=\delta^0(E)-\delta^1(E)+1\\
0=h^0H^1(\PP^1,j_*\cE)&=-\delta^0(E)-h^0(E)=-\delta^0(E)-1.\qedhere
\end{align*}
\end{proof}

\begin{proof}[\proofname\ of Theorem \ref{thmhodgeregular}]
From Proposition \ref{prop:numdataMCchi} one derives that $\MC_{-1}(L_0)$ is a variation of rank $2$ whose local monodromy at $x_1,x_2,x_3$ is indecomposable unipotent and is the scalar matrix $-1_2$ at $\infty$. Since $h^0(L_0)=1$, $\delta^0(L_0)=-2$ and $\mu_{x_i,-1}^0(L_0)=1$ (and all other Hodge data are zero), Formula \eqref{eq:nupH1*} gives $h^0\MC_{-1}(L_0)=h^1\MC_{-1}(L_0)=1$.

Recall that, for a variation of polarized Hodge structure $(V,F^\cbbullet V,\nabla,k)$, $V\otimes V$, $\Sym^2V$ and $\Lambda^2V$ also underlie such a variation and, setting $\gr_F^{p/2}V=0$ if $p/2\not\in\ZZ$, we have natural isomorphisms
\begin{align}
\gr^p_F(V\otimes V)&\simeq\bigoplus_{j+k=p}(\gr^j_FV\otimes\gr^k_FV),\notag\\
\gr^p_F(\Sym^2V)&\simeq\bigoplus_{\substack{j<k\\j+k=p}}(\gr^j_FV\otimes\gr^k_FV)\oplus\Sym^2(\gr^{p/2}_FV),\notag\\\label{eq:antisym}
\gr^p_F(\Lambda^2V)&\simeq\bigoplus_{\substack{j<k\\j+k=p}}(\gr^j_FV\otimes\gr^k_FV)\oplus\Lambda^2(\gr^{p/2}_FV).
\end{align}
We will use these formulas when $\rk\gr^k_FV=1$ for all $k$, so that we will replace $\Sym^2(\gr^{p/2}_FV)$ by $\gr^{p/2}_FV$ and $\Lambda^2(\gr^{p/2}_FV)$ by $0$.

It follows that the symmetric square $\Sym^2(\MC_{-1}(L_0))$ is a variation of rank~$3$ whose local monodromy at $x_1,x_2,x_3$ is easily seen to be indecomposable unipotent (and identity at~$\infty$). It is irreducible (a rank-one sub or quotient  local system would come from a rank-one sub or quotient local system of $\MC_{-1}(\cL_0)$ by the diagonal embedding, according to the form of the local monodromies). It follows from Lemma~\ref{lemme4.0.1} (by~shifting the singularities $0,1,\infty$ to $x_1,x_2,x_3$, respectively), that
$$
\delta^2(\Sym^2(\MC_{-1}(L_0)))=1,\ \delta^1(\Sym^2(\MC_{-1}(L_0)))=0,\ \delta^0(\Sym^2(\MC_{-1}(L_0)))=-1.
$$
The local Hodge data of $L_1\otimes \Sym^2(\MC_{-1}(L_0))$ are as follows:
$$
h^0=h^1=h^2=1,\quad
\mu_{x_i,1,1}^2=1\, (i=1,2),\quad \mu_{x_3,-1,2}^2=1,
$$
and all other $\mu_{x_i,\lambda,\ell}^p$'s vanish, according to \eqref{eq:nupVL}--\eqref{eq:muVL}. By Proposition \ref{prop:degreetensor}, the Hodge degrees of $L_1\otimes \Sym^2(\MC_{-1}(L_0))$ are
$$
\delta^2=0,\quad \delta^1=-1,\quad \delta^0=-2.
$$

Set $V=\MC_{-1}(L_1\otimes \Sym^2(\MC_{-1}(L_0))$. By Theorem \ref{th:MCchi}(1) and \eqref{eq:nupH1*}, we have
$$
h^p(V)=\delta^{p-1}-\delta^p-h^p+\sum_{i=1}^r(\nu_{x_i,\neq 1}^{p-1}+\mu_{x_i,1}^p).
$$
Hence $ h^p(V)=1$ for $p=0,1,2,3$. By Proposition \ref{prop:numdataMCchi}, the local monodromy at $x_1,x_2,x_3,x_4$ of $V$ is:
$$
1_2\oplus -\rJ(2), \quad
1_2\oplus -\rJ(2), \quad
\rJ(4),\quad -1_4.
$$
By Theorem \ref{th:MCchi}(2), the local Hodge data of $V$ are as follows: For $i=1,2$ we have $\mu_{x_i,-1,1}^2=1$ and $0$ else, so from Remark \ref{rem:missingnup} we obtain $\nu_{x_i,1,0}^0=\nu_{x_i,1,0}^3=1$, and $\mu_{x_3,1,2}^3=1$ and $0$ else. Using Theorem \ref{th:MCchi}(3), one obtains the following global Hodge data $\delta^p=\delta^p(V)$:
\begin{equation}\label{eq:degV}
\delta^0=-1,\quad \delta^1=-2,\quad \delta^2=-1,\quad \delta^3=0.
\end{equation}

Let us now consider $\wt\Lambda^2V$, which has rank $5$. We will need a specific argument to compute its local and global Hodge data. From \eqref{eq:antisym} we obtain isomorphisms over $\PP^1\setminus\nobreak\{x_1,\dots,x_4\}$:
\begin{equation}\label{eq:grFV}
\gr_F^p(\Lambda^2(V))\simeq
\begin{cases}
(\gr_F^j(V)\otimes\gr_F^k(V))\quad\text{for }\begin{cases}
p\neq3,\\0\leq j<k\leq3,\\j+k=p,\end{cases}\\
(\gr_F^1(V)\otimes\gr_F^2(V))\oplus(\gr_F^0(V)\otimes\gr_F^3(V))\quad\text{for }p=3.
\end{cases}
\end{equation}

Locally at $x_i$ ($i=1,\dots,4$), let $e_0^i,\dots,e_3^i$ be a basis of the canonical extension~$V^0$ of $V$ lifting a basis $\wt e_0^i,\dots,\wt e_3^i$ of $V^0/V^1$ such that $\wt e_k^i$ induces a basis of the rank-one vector space $\gr^k_F(V^0/V^1)$. We will use below the identification $\psi_{x_i,-1}(V)=\gr^{1/2}_V(V)$ instead of $\gr^{-1/2}_V(V)$ (this is equivalent, even from the point of view of the Hodge filtration, according to the definitions in \S\ref{subsec:localHodgetheory}). Hence, locally, we have a basis of $\Lambda^2(V)^{-\infty}$
given by (the anti-symmetrization of)
$$
v_1^i=e_0^i\otimes e_1^i,\
v_2^i=e_0^i\otimes e_2^i, \
v_3^i=e_1^i\otimes e_2^i,\
v^{\prime i}_3=e_0^i\otimes e_3^i,\
v_4^i=e_1^i\otimes e_3^i,\
v_5^i=e_2^i\otimes e_3^i.
$$

At $x_3$, we can choose $\wt\bme{}^3$ such that $t\partial_t\wt e_j^3=\wt e_{j-1}^3$, which implies the following properties (locally at $x_3$):
\begin{itemize}
\item
the local monodromy of $\wt \Lambda^2(V)$ is $\rJ(5)$, with basis $v_5^3,v_4^3,v^{\prime 3}_3+v_3^3,2v_2^3,2v_1^3$, and the direct summand $E$ of rank one corresponding to the symplectic form has basis $v^{\prime 3}_3-v_3^3$,
\item
we have $ \mu_{x_3,1,3}^5(\wt \Lambda^2(V))=1$ and all other $\mu_{x_3,\lambda,\ell}^p$ vanish,
\item
Formula \eqref{eq:grFV} extends to a similar formula for the local Hodge bundles $\gr^p_FV^0(\Lambda^2(V))$,
\end{itemize}
and then, over $\PP^1\setminus \{x_1,\dots,x_4\}$,
\begin{itemize}
\item
we have $E=\gr^3_FE$,
\item
we have $h^1(\wt\Lambda^2V)=\cdots=h^5(\wt\Lambda^2V)=1$.
\end{itemize}

Let us now consider the local situation at $x_i$, $i=1,2$. From our computation of $\mu_{x_i,-1,\ell}^p(V)$, we can choose the basis $\wt\bme{}^i$ so that the nilpotent operator $(t\partial_t-\nobreak\frac{1}{2}):\psi_{x_i,-1}(V)\to \psi_{x_i,-1}(V)$ acts as~\hbox{$\wt e_2^i\mto\wt e_1^i\mto0$}. We also have $t\partial_t\wt e_0^i=t\partial_t\wt e_3^i=0$. Therefore, the corresponding operator $(t\partial_t-\frac{1}{2})$ on $\psi_{x_i,-1}(\Lambda^2(V))$ acts as \hbox{$\wt v_2^i\mto\wt v_1^i\mto0$} and $\wt v_5^i\mto\wt v_4^i\mto0$. In a similar way one checks that $v_3^i$ belongs to $V^1(\Lambda^2V)$, while $v_3^{\prime i}$ belongs to $V^0(\Lambda^2V)$. We also note that $E=\gr^3_FE$ (seen above) is generated by some combination of $v_3^i$ and $v^{\prime i}_3$. Moreover,
\begin{itemize}
\item
the local monodromy of $\wt \Lambda^2(V)$ is $-\rJ(2)\oplus-\rJ(2)\oplus1$,
\item
the local Hodge data are: $ \mu_{x_i,-1,1}^2(\wt \Lambda^2(V))=\mu_{x_i,-1,1}^5(\wt \Lambda^2(V))=1$ and all other $\mu_{x_i,\lambda,\ell}^p$ vanish,
\item
the first line ($p\neq3$) of Formula \eqref{eq:grFV} extends locally at $x_i$ as it is to the Hodge bundles $\gr^pV^0(\Lambda^2(V))$, while the second line extends as
\[
\gr^3_FV^0(\Lambda^2(V))=x^{-1}(\gr_F^1(V^0)\otimes\gr_F^2(V^0))\oplus(\gr_F^0(V^0)\otimes\gr_F^3(V^0)),
\]
where $x$ is a local coordinate at $x_i$.
\end{itemize}

Lastly, at $x_4$, $\wt v_1,\dots,\wt v_5$ form a basis of $\gr^1_V( \Lambda^2(V))$. We conclude that
\begin{itemize}
\item
the local monodromy of $\wt \Lambda^2(V)$ is $1_5$,
\item
we have $\mu_{x_4}=0$,
\item
Formula \eqref{eq:grFV} extends to a similar formula for the local Hodge bundles $\gr^p_FV^1(\Lambda^2(V))=x\gr^p_FV^0(\Lambda^2(V))$.
\end{itemize}

We can now compute the degree of $\gr^p_FV^0(\Lambda^2(V))$ by using the local shift information comparing this bundle with $\gr^j_F\gr^0_V(V)\otimes\gr^k_F\gr^0_V(V)$. We find:
\begin{align*}
\gr^{j+k}_FV^0(\Lambda^2(V))&=\big(\gr^j_F(V^0)\otimes\gr^k_F(V ^0)\big)(1)\quad\text{if $j+k\neq3$ (shift at $x_4$)},\\
\gr^3_FV^0(\Lambda^2(V))&=\big(\gr^1_F(V^0)\otimes\gr^2_F(V^0)\big)(3)\oplus\big(\gr^0_F(V^0)\otimes\gr^3_F(V^0)\big)(1)\\&\hspace*{6cm}\text{(shift at $x_1,x_2,x_4$)}.
\end{align*}
Taking now into account \eqref{eq:degV} we find, for $p=1,\dots,5$ respectively,
$$
\delta^p(\Lambda^2(V))=-2,-1,0,-1,0.
$$
We have $\delta^p(\wt\Lambda^2(V))=\delta^p(\Lambda^2(V))$ for $p\neq3$. On the other hand, $E=\gr^3_FE$ is a direct summand of $\gr^3_F(\Lambda^2(V))\simeq\cO_{\PP^1}^2$ by the previous degree computation. It follows that both $E$ and $\gr^3_F(\wt\Lambda^2(V))$ are isomorphic to $\cO_{\PP^1}$ and $\delta^3(\wt\Lambda^2(V))=0$.

We can now continue the computation as in \S\ref{subsec:exam1} in order to obtain the desired table.
\end{proof}

Let us now specialize the points $x_1,x_2,x_3$ to $0,1,-1$ (respectively). Note that for each prime number $\ell$ there are lisse \'etale $\QQ_\ell$-sheaves $\cL_{i,\ell}$, $i=0,1,2,3$ of rank one on $\Afu_{\ZZ[\sfrac{1}{2}]}\setminus \{0,1,-1\}=\Spec\left(\ZZ[\frac{1}{2}][x, \frac{1}{x},\frac{1}{x-1},\frac{1}{x+1}]\right)$ whose analytifications have a $\ZZ$-form given by $\cL_i$, $i=0,1,2,3$ (respectively). One may construct~$\cL_{0,\ell}$ by considering the \'etale Galois cover $g:X\to \Afu_{\ZZ[\sfrac{1}{2}]}\setminus \{0,1,-1\}$ defined by the equation
$$
y^2=x(x-1)(x+1).
$$
Denote the automorphism $y\mto -y$ of $X$ by $\sigma$. Then the sheaf $\cL_{0,\ell}:=\frac{1}{2}(1-\nobreak\sigma)g_*(\QQ_{\ell,X})$ has the desired properties. The sheaves $\cL_{i,\ell}$ ($i=1,2,3$) can be constructed in the same way. Define
$$
\cH_\ell= \MC_{-1}(\cL_{3,\ell}\otimes \MC_{-1}(\cL_{2,\ell}\otimes \tilde{\Lambda}^2(\MC_{-1}(\cL_{1,\ell}\otimes \Sym^2(\MC_{-1}(\cL_{0,\ell})))) )),
$$
where $\MC_{-1}$ is as in \cite[Chap\ptbl4 and Chap\ptbl8]{Katz96}. The monodromy representation of~$ \cH_\ell$
$$
\rho_\ell:\pi_1(\Afu_{\ZZ[\sfrac{1}{2}]}\setminus \{0,1,-1\})\to \GL_7(\QQ_\ell)
$$
takes values in $G_2(\QQ_\ell)\times \text{Center}(\GL_7(\QQ_\ell))$. For any $\QQ$-point $s$ of $\Afu_{\ZZ[\sfrac{1}{2}]}\setminus \{0,1,-1\}$ the composition of $\rho_\ell$ with the map of \'etale fundamental groups $G_\QQ:=\Gal(\ov\QQ/\QQ)\to \pi_1(\Afu_{\ZZ[\sfrac{1}{2}]}\setminus \{0,1,-1\})$ induced by $s$ by functoriality defines a rank-$7$ Galois representation
$$
\rho_\ell^s: G_\QQ\to \GL_7(\QQ_\ell),
$$
the {\it specialization of $\rho_\ell$ at $s$.} A Galois representation $\kappa:G_\QQ\to \GL_m(\QQ_\ell)$ is called {\it potentially automorphic} if there exists a finite field extension $K/\QQ$ and an automorphic representation $\pi$ of $\GL_m(\AA_K)$ ($\AA_K$ denoting the adèles of $K$) such that the $L$\nobreakdash-functions of $ \kappa$ and of $\pi$ match (up to finitely many factors). The following result can hence be interpreted as a weak version of Langlands conjectures on the automorphy of the Galois representations $\rho_\ell^s$:

\begin{theoreme}\label{thmpotaut}
For all $\QQ$-rational points $s$ of $\Afu_{\ZZ[\sfrac{1}{2}]}\setminus \{0,1,-1\}$ the specializations $\rho_\ell^s:G_\QQ\to \GL_7(\QQ_\ell)$ are potentially automorphic. In particular, the associated partial $L$-function has a meromorphic continuation
to the complex plane and satisfies the expected functional equation.
\end{theoreme}

\begin{proof}
It follows from \cite[Th\ptbl5.5.4]{Katz96} that for a fixed $\QQ$-rational point $s\neq -1,0,1$, the collection of Galois representations $(\rho_\ell^s)_{\ell \textrm{ prime}}$ has the property that for $p$ outside a finite set of primes $S$, the characteristic polynomial of the Frobenius elements at $p$ is an element in $\QQ[t]$ which is independent of $\ell.$ Moreover, each $\rho_\ell^s $ is pure of weight $8$ for each Frobenius element belonging to a prime outside $S\cup\{\ell\}$ (since the tensor operations preserve purity by doubling the weight and each middle convolution step increases the weight exactly by~$1$). By the motivic nature of $\MC_\chi$ and of tensor products, the system $(\rho_\ell^s)_\ell$ occurs in the cohomology of a smooth projective variety defined over $\QQ$ (use equivariant desingularization and the arguments from \cite{D-R10}, Section~2.4). In fact, it follows from the same arguments as in \cite[Section 3.3]{D-R10} that $\rho_\ell^s$, inside the $\ell$-adic cohomology of the underlying variety, is the kernel of a restriction morphism, cut further out by out by an algebraic correspondence associated to an involutive automorphism. Hence, for $s$ fixed, the restricted Galois representations $(\rho_\ell^s|_{\Gal(\ov\QQ_\ell/\QQ_\ell)})_\ell$ are crystalline at places of good reduction and de Rham elsewhere (this follows from the work of Faltings \cite{Faltings89} and is also implied by more general results of Tsuji \cite{Tsuji99}). The above can be summarized by saying that the system $(\rho_\ell^s)_{\ell}$ is weakly compatible and pure of weight~$8$ in the sense of~\cite{B-G-G-T10} and \cite{PatrikisTaylor}. It follows inductively from Poincar\'e duality and from the fact that the weight is even that the Galois representations $\rho_\ell^s$ factor through a map to the group of orthogonal similitudes ${\rm GO}_7(\QQ_\ell)$ with an even multiplier, \ie the system $(\rho_\ell^s)_{\ell}$ is essentially self\-dual in the sense of \cite{PatrikisTaylor}. We claim that the system is also Hodge-Tate regular, \ie the non-vanishing Hodge-Tate numbers
\[
\dim_{\ov\QQ_\ell}\gr^i(\rho_\ell^s\otimes B_{\DR})^{\Gal(\ov\QQ_\ell/\QQ_\ell)}
\]
(where $B_{\DR}$ denotes the usual de Rham ring) are all equal to $1$ (\cf \cite{B-G-G-T10}, \cite{PatrikisTaylor}): it follows from the de Rham comparison isomorphism (which is compatible with tensor products, filtrations, cycle maps and Galois operation (Faltings, \loccit)) and from Theorem~\ref{thmhodgeregular} that
\[
\dim_{\ov\QQ_\ell}\gr^i(\rho_\ell^s\otimes B_{\DR})^{\Gal(\ov\QQ_\ell/\QQ_\ell)}= \dim\gr^{-i}_F H_s\leq 1,
\]
as claimed. Taking the above properties together we conclude by \cite[Th\ptbl A and Cor\ptbl C]{PatrikisTaylor}.
\end{proof}

\backmatter
\newcommand{\SortNoop}[1]{}\def\cprime{$'$}
\providecommand{\bysame}{\leavevmode\hbox to3em{\hrulefill}\thinspace}
\providecommand{\MR}{\relax\ifhmode\unskip\space\fi MR }
\providecommand{\MRhref}[2]{%
  \href{http://www.ams.org/mathscinet-getitem?mr=#1}{#2}
}
\providecommand{\href}[2]{#2}

\end{document}